\documentclass[12pt,leqno]{amsart} 
\allowdisplaybreaks
\usepackage{amsmath}
\usepackage{mathtools}
\usepackage{amssymb}
\usepackage{amsthm} 
\usepackage{epsf}
\usepackage{graphicx}
\usepackage{dsfont}
\usepackage{hyperref} 
\usepackage{esint}
\usepackage{enumerate}
\usepackage[shortlabels]{enumitem}
\usepackage{comment}
\hypersetup{pdfpagemode=FullScreen,  colorlinks=true} 
\usepackage[english]{babel}
\usepackage{tikz}

\usetikzlibrary{calc,fadings,decorations.pathreplacing}

\numberwithin{equation}{section}

\newtheorem{lemma}{Lemma}[section]

\newtheorem{defn}[lemma]{Definition}
\newtheorem{thm}[lemma]{Theorem}
\newtheorem{cor}[lemma]{Corollary}
\theoremstyle{remark}
\newtheorem{rmk}[lemma]{\bf Remark}
\theoremstyle{definition}

\newcommand{\N}{\mathbb{N}}
\newcommand{\Z}{\mathbb{Z}}

\newcommand{\R}{\mathbb{R}}
\newcommand{\C}{\mathbb{C}}
\newcommand{\ra}{\rightarrow}
\newcommand{\da}{\downarrow}

\newcommand{\HD}{\mathcal{H}}
\newcommand{\dist}{\text{dist}}
\newcommand{\ntlim}{\text{n.t.} \lim}
\newcommand{\sphere}{\mathbb{S}}

\title{Non-local distance functions and geometric regularity}
\author{Max Engelstein, Cole Jeznach, and Svitlana Mayboroda}
\thanks{M. Engelstein was partially supported by NSF DMS 2000288. C. Jeznach was partially supported by the Simons Collaborations in MPS grant 563916, SM and M.E.'s NSF grant 2000288. S. Mayboroda was partially supported by Simons Collaborations in MPS grant 563916, SM, and NSF DMS 1839077. Part of the writing of this paper took place while the second and third authors were in residence at the Hausdorff Research Institute for Mathematics (HIM) in Bonn, Germany, during the winter trimester program of 2022. They thank the Institute for its hospitality. Finally, the authors would like to thank Dmitriy Bilyk for pointing out Theorem 4.1 from \cite{HELGASON} which enabled us to prove Corollary \ref{rmk:dist_exact_codim_2}.
}
\subjclass[2010]{Primary: 28A75, 35J70. Secondary: 42B20}
\keywords{Regularized distance, uniform rectifiability, square function estimates} 
\address{School of Mathematics, University of Minnesota,University of Minnesota, Minneapolis, MN, 55455, USA.}
\email{mengelst@umn.edu} \email{jezna001@umn.edu} \email{svitlana@math.umn.edu}

\date{\today}

\begin{document}

\begin{abstract}
We establish the equivalence between the regularity (rectifiability) of sets and suitable estimates on the oscillation of the gradient for smooth non-local distance functions. A prototypical example of such a distance was introduced, as part of a larger PDE theory, by David, Feneuil, and Mayboroda in \cite{DFM18}. The results apply to all dimensions and co-dimensions, require no underlying topological assumptions, and provide a surprisingly rich class of analytic characterizations of rectifiability. 
\end{abstract}

\maketitle

\tableofcontents


\section{Introduction}
\label{sec:introduction}

Some of the major efforts at the interface between geometric measure theory, analysis, and partial differential equations in the past 20-30 years have been devoted to the characterization of regularity of sets in terms of their PDE and analytic properties. This program, starting with the work of F. and M. Riesz in the plane \cite{RR} in 1916, has recently achieved the characterization of uniform rectifiability in terms of the absolute continuity of harmonic measure \cite{AHMMT20} and the solution of the David-Semmes conjecture, a characterization of uniform rectifiability in terms of the Riesz transform \cite{NTVRIESZ1, MMV96}. One of the key achievements on this path was undoubtedly the introduction of the ``correct" notion of regularity, uniform rectifiability, amenable to such scale-invariant characterizations, by David and Semmes in the early 90s, along with their first characterization, the so-called Uniform Square Function Estimate (USFE) in \cite{DSUR}. All in all such characterizations are still scarce, notoriously challenging, and many are compelling open problems. 

Aside from the Riesz transform, very few singular integral operators are known to characterize uniform rectifiability of $(n-1)$-dimensional sets  (see, e.g. \cite{Merchan} for a non-perturbative example in the plane and \cite{PPT} for some perturbative results).  The counterexamples are rare as well, but it is known that for some Calder\'on-Zygmund kernels the existence of principle values or the $L^2$-boundedness of the associated operator does not characterize uniform rectifiability (cf. \cite{HUOVINEN01}, \cite{DCANTOR01}, and \cite{MSURVEY97}).  No characterization of uniform rectifiability using square functions  pertaining to different kernels are available to date, though there are some results on the boundedness of square functions for other kernels assuming uniform rectifiability (e.g. \cite{HOFMANNMEMOIRS} and \cite{MMPAMS}).

Moreover, in the context of lower dimensional sets, these questions are of great interest and almost completely open.  On the singular integral side,  Jaye and Nazarov have introduced a beautiful approach via reflectionless measures, aiming to extend the singular integral characterizations  to any dimension and co-dimension  (see, e.g. \cite{JN1, JN2}), but unfortunately, checking the initial condition of their theory even for the Riesz transform is still unattainable.  As a result, the Riesz transform characterization of $d$-rectifiability, $1 < d<n-1$, remains one of the outstanding open problems, and the straightforward analogue of the USFE is known to fail \cite{DSUR}. The analogue of the harmonic measure characterization for $d<n-1$ fails too, and in fact, the only known PDE characterization of uniform rectifiability of lower dimensional sets is the recently obtained estimate in terms of the appropriate Green functions in \cite{DMGREEN}. 

The goal of the present paper is a new characterization of uniform rectifiability in terms of the generalized regularized distance function. On one hand, it could be seen as parallel to the Riesz transform characterization in \cite{NTVRIESZ1} or even in \cite{TPV08}, since a non-tangential limit of the regularized distance function formally, for some values of parameters, looks like the Riesz transform, although in the full generality it is non-linear, non-local, and not a traditional singular integral operator.  On the other hand, the theory developed here is parallel to the USFE, although once again, the resemblance is formal. This theory also, quite magically, connects to PDEs: the regularized distance function is the Green function with a pole at infinity for a certain special degenerate differential operator \cite{DEMMAGIC}, not to mention its major role in the newly emerged elliptic theory for domains with lower dimensional boundaries (see, e.g., \cite{DFM18, DMURAINFTY, DMGREEN}). One could also say that it is an alternative characterization altogether: the oscillations of the gradient of the regularized distance are more reminiscent of curvature than any of the above. Most importantly, it applies to all dimensions and co-dimensions. 

The first step in this direction was taken by the first and the third author together with Guy David in 2018 \cite{DEMMAGIC}, where the appropriate characterizations in terms of  
\begin{equation}
D_{\mu, \alpha}(x)  \equiv \left( \int \dfrac{1}{|x-y|^{d+\alpha}} \; d\mu(y) \right)^{-1/\alpha}, \; \; x \in \R^n, \label{def_reg_dist}
\end{equation}
were achieved. The notion of $D_{\mu, \alpha}$ itself was first introduced in \cite{DFM18} by David, Feneuil, and the third author. If $\mu$ is $d$-Ahlfors regular (cf. \eqref{e:AR}), then $D_{\mu, \alpha}$ is smooth away from $E$ but also acts like a distance to the support of $\mu$ in the sense that $D_{\mu,\alpha}(x) \simeq \mathrm{dist}(x, \mathrm{spt }\, \mu)$. As we alluded to above, in \cite{DEMMAGIC}  it was also shown that for a special ``magic" value of $\alpha$, the distance given by \eqref{def_reg_dist} is in fact the Green function with pole at infinity for $- \text{div} \left(D_{\mu, \alpha}^{-n+d+1} \nabla \; \cdot \right)$, outside of {\it any} Ahlfors regular set. The consequences of this fact are numerous and powerful and still being explored (see, e.g. \cite{F21AC, Polina}). Relatedly, applications of these regularized distances to free boundary problems have recently been discovered in \cite{DESVGT}.

The present paper has started with the natural question for which other kernels $K$ one can characterize rectifiability of a $d$-dimensional set, $d\leq n-1$, using a generalized distance 
\begin{equation}\label{e:regdistk}D_{K, \mu,\alpha}(x) \equiv R_{K, \mu, \alpha}(x)^{1/\alpha}  \equiv \left(\int \dfrac{K(x-y)}{|x-y|^{d+\alpha}} \; d\mu(y)\right)^{-1/\alpha}. 
\end{equation}
However, as we hopefully described above, such questions are far from innocent. The exact properties of the kernel needed for a {\it characterization} are extremely delicate, which is why very few substitutes for the classical Riesz transform are known (e.g. \cite{Merchan, PPT}), and none of them applies to lower dimensional sets. This is also why it is so hard to check that a given operator satisfies the reflectionless condition of Jaye and Nazarov. Moreover, we cannot even draw an analogy with these few ``good" singular integral operators as the cancellations of the kernel responsible for quantifying the geometry of the sets by singular integrals are very different from the behavior our kernels in our distance functions, which are necessarily non-degenerate, hence, emphatically avoiding cancellations. Yet, the present paper achieves a rich and comprehensive theory.

In this paper, to our surprise, and in contrast to the situation for Calder\'on-Zygmund kernels, in every dimension and co-dimension we produce examples of kernels $K$, which are not perturbations of constants, such that the oscillation of $|\nabla D_{K,\mu,\alpha}|$ characterizes uniform rectifiability. More precisely, using a novel functional-analytic argument, we are able to give examples of distance functions that are equal to the regularized distance $D_{\mu, \alpha}$ whenever $\mu$ is flat, but which a priori may act very differently outside of general measures, see Theorem~\ref{t:main} below. In addition, we establish a perturbative theory, showing that if $K$ is close to $K'$, a kernel with good behavior outside of flat sets, then the oscillation of $|\nabla D_K|$ also characterizes geometric regularity. A complete description of the kernels for which the oscillation of $|\nabla D_K|$ characterizes geometric regularity  is complicated by the aforementioned construction of a large family of ``good" kernels $K'$ (see, e.g. Theorem \ref{t:nonradialnt}).  Nonetheless, in the radial setting we are indeed able to establish a description of {\it all} kernels that characterize rectifiability -- once again, note the difference with the singular integral operator results where only a few examples and counterexamples are available. 

To more precisely discuss our work, we now introduce some definitions.  

In what follows, we always take $\mu$ to be a $d$-Ahlfors regular measure on $\R^n$ with $0 < d < n$, not necessarily an integer. That is, $\mu$ is a measure for which there is a uniform constant $C>0$ such that
\begin{align}\label{e:AR}
C^{-1} R^d & \le \mu(B(Q,R)) \le C R^d,
\end{align}
for each $Q \in \text{spt }\mu$ and every $R >0$. Given such a $\mu$, a number $\alpha > 0$ and a function $K \in C(\mathbb R^n \setminus \{0\})$ we define the regularized distance to $\mathrm{spt}\, \mu =: E$ according to formula \eqref{e:regdistk} above. In particular,  
\begin{equation}\label{e:rkdef}
R_K(x) = R_{K, \mu, \alpha}(x) = \int \dfrac{K(x-y)}{|x-y|^{d+\alpha}} \; d\mu(y),\end{equation} so that $D_{K,\mu,\alpha} = R_{K,\mu,\alpha}^{-1/\alpha}$. When $K \equiv 1$ these are the regularized distances \eqref{def_reg_dist} introduced in \cite{DFM18} and further studied in \cite{DEMMAGIC}. In the latter work it was important that $D_{1, \mu,\alpha}(x) \simeq \mathrm{dist}(x, E)$ and that $D$ was smooth on $\mathbb R^n \backslash E$ with appropriate estimates. To guarantee that those properties also hold for $D_{K, \mu,\alpha}$ we impose the following conditions on the kernel $K$.

\begin{defn}
We say that a positive function $K \in C^2(\R^{n} \setminus \{0\})$ is a distance-standard kernel if 
$$
    \|  \nabla^m K(x) |x|^m  \|_\infty < \infty, \quad \mbox{for} \quad m = 0, 1, 2,
$$
and 
$$\inf_{x \in \R^n \setminus\{0\} } K(x) > 0.$$
For such functions, we say that the distance-standard constant associated to $K$ is 
\begin{align*}
    \max \{ \|K\|_\infty, \|\nabla K(x) |x| \|_\infty, \|\nabla^2 K(x) |x|^2 \|_\infty, \|1/K\|_\infty \}.
\end{align*}
\end{defn}
Using a dyadic shells argument one can see that $D_{K, \mu, \alpha}(x) \simeq \mathrm{dist}(x, E)$ with constants depending on $n, d, \alpha$, the Ahlfors regularity constant of $\mu$ and the distance-standard constant of $K$. Furthermore, one can differentiate under the integral to show that $D_{K, \alpha, \mu} \in C^2(\mathbb R^n \backslash E)$. To simplify notation we write $\Omega = \R^n \setminus E$ and denote by $\delta_E(x) = \dist(x ,E)$ the Euclidean distance to the set $E$. Also, we often drop the dependence of $D_{K, \mu, \alpha}$ on $\mu$ and $\alpha$ when clear from context, and instead write $D_K$ or $D_{K, \mu}$. Respectively, we often denote the original regularized distance $D_{\mu, \alpha}$ by $D_1$ or $D_{1, \mu}$. Whenever $E \subset \R^n$ is a $d$-plane, we take $\mu = \HD^d|_E$ unless otherwise specified.

As mentioned above, we hope to characterize geometric regularity by the oscillation of $|\nabla D_{K,\mu, \alpha}|$. Following \cite{DEMMAGIC} we measure this oscillation in two ways. The first one pertains to the existence of non-tangential limits.
\begin{defn}\label{ntlimits}
For $Q \in E, R >0$ and $\eta \in (0, 1)$ we let 
\begin{align*}
    \Gamma_{R, \eta}(Q) := \{ x \in \Omega \cap B(Q,R) \; ; \; \text{dist}(x, E) \ge \eta|x-Q| \}.
\end{align*}
We say that $f$ has a non-tangential limit, $L$, at $Q \in E$ if there is some $\eta \in (0,1)$ such that 
\begin{align*}
    \lim_{R \da 0} \sup_{x \in \Gamma_{R, \eta}(Q)} |f(x) - L| = 0.
\end{align*}
In this setting, we write $\ntlim_{x \ra Q}^\eta f(x) = L$.
\end{defn}

To address the second one, we introduce 
\begin{align}
F_K(x) \equiv F_{K, \mu, \alpha}(x) & := \delta_E(x) \left | \nabla \left| \nabla D_K(x) \right|^2 \right|, \; x \in \Omega.
\end{align}
The quantity $F_K$ measures the oscillation of $|\nabla D_K|$ in a scale-invariant way. For a general $d$-Ahlfors regular measure $\mu$, $F_K$ is merely a bounded continuous function. One of the main results of this paper is that regularity (uniform rectifiability) of the boundary of $\Omega$ is equivalent to an enhanced estimate controlling oscillations of $|\nabla D_K|$ through the following Carleson measure condition on $F_K$:
 \begin{equation}\label{e:CMC}
 \sup_{Q\in E} \sup_{r > 0} \frac{1}{r^d}\int_{B(Q,r)\cap \Omega} F_K(x)^2 \delta_E(x)^{-n+d}\, dx < \infty.
 \end{equation}
 
Following \cite{DEMMAGIC} and inspired by \cite{DSUR} we refer to condition \eqref{e:CMC} as the USFE (usual square function estimate). 

Let us now carefully define rectifiability and uniform rectifiability.
\begin{defn}
A set $E \subset \R^n$ is said to be $d$-rectifiable for $d \in \N$ if there exist countably many Lipschitz maps $f_j: \R^d \ra \R^n$ such that 
\begin{align*}
\HD^d\Bigl( E \setminus \bigcup_j f_j(\R^d)\Bigr) & = 0.
\end{align*}
If $\mu$ is a Radon measure on $\R^n$, then we say that $\mu$ is $d$-rectifiable if $\mu \ll \HD^d$ and there is a $d$-rectifiable Borel subset $E \subset \R^n$ with $\mu(\R^n \setminus E) =0$. 
\end{defn}
\begin{defn}
A $d$-Ahlfors regular set $E \subset \R^n$ is said to be $d$-uniformly rectifiable for $d \in \N$ if there exist uniform constants $M, \theta >0$ such that for each $Q \in E$ and each $R >0$ there is a Lipschitz map $f:B(0,R) \subset \R^d \ra \R^n$ with Lipschitz norm $\le M$ such that
\begin{align*}
\HD^d( E \cap B(Q, R) \cap f(B(0,R))) \ge \theta R^d.
\end{align*}
If $\mu$ is $d$-Ahlfors regular measure on $\R^n$, then we say that $\mu$ is $d$-uniformly rectifiable if its support is $d$-uniformly rectifiable.
\end{defn}

Theorem 1.5 in \cite{DEMMAGIC} says that a set $E$ is rectifiable if and only if $|\nabla D_{1, \mathcal H^d|_E, \alpha}|$ has non-tangential limits almost-everywhere on $E$. One could draw the aforementioned parallels between this result and the Riesz transform characterization, in particular, \cite{TPV08}, as 
$$ \nabla D_{1, \mathcal H^d|_E, \alpha}= -\frac 1\alpha \left( \int \dfrac{1}{|x-y|^{d+\alpha}} \; d\mu(y) \right)^{-1/\alpha-1} \int \dfrac{x-y}{|x-y|^{d+\alpha+2}} \; d\mu(y). $$
Setting formally $\alpha=-1$ above and properly re-interpreting the integrals would transform the latter term into the classical Riesz transform. However, our $\alpha$ is always a positive number, so that the resultant expression, while analogous, is actually a quite surprising extension of the concept of the Riesz transform  (note that the expression for $\alpha >0$ is nonlinear, and does not represent a Calder\'on-Zygmund singular integral). In a similar vein, inspired by work of David and Semmes  \cite{DSUR} on square functions, \cite[Theorem 1.4]{DEMMAGIC} states that a set $E$ is uniformly rectifiable if and only if $F_{1,\mathcal H^d|_E, \alpha}$ satisfies a Carleson measure estimate outside of $E$. 

\vskip 0.08in

Here, we ask for which $K$ do the following hold:
\begin{align}
& \mu \text{ is } d\text{-rectifiable} \text{ if and only if the non-tangential limits of } |\nabla D_K| \text{ exist $\mu$ a.e. in }E, \label{intro:q1} \\
&\mu \text{ is } d\text{-uniformly rectifiable} \text{ if and only if } D_K \text{ satisfies the USFE on } \Omega \label{intro:q2}. 
\end{align}
We first show in Section \ref{sec:class} that the answers to \eqref{intro:q1} and \eqref{intro:q2} are both yes whenever \begin{equation}\label{e:distexact} D_{K, \mathcal H^d|_E, \alpha}(x) = c \delta_E(x), \qquad \forall x\in \mathbb R^n\backslash E, \qquad \forall E \in G(n,d).\end{equation} 
Here, and throughout the paper, we use $G(n,d)$ to denote the Grassmannian of $d$-dimensional planes through the origin in $\R^n$. We use $A(n,d)$ to denote $d$-dimensional affine sets in $\R^n$. 

Our main result is that there is a large family of distance standard kernels which satisfy the above relation:

\begin{thm}[Main Theorem]\label{t:main}
For each pair of integers $d < n$, and every $\alpha > 0$ there exists a non-constant smooth distance-standard kernel, $K$, which satisfies \eqref{e:distexact}. In particular, the  characterizations of rectifiability by \eqref{intro:q1} and the characterization of uniform rectifiability by \eqref{intro:q2} are both valid for such $K$.

Furthermore, $K$ may be chosen to be far from being constant in the sense that $K_\lambda(x) := K(\lambda x)$ converges to a non-constant kernel in $C^1_{\mathrm{loc}}(\mathbb R^n \backslash 0)$ as $\lambda \downarrow 0$. 

On the other hand if $K$ is invariant under rotations then $K$ satisfies \eqref{e:distexact} if and only if $K$ is a constant.
\end{thm}

We are then able to show that the answers to questions \eqref{intro:q1}, \eqref{intro:q2} are yes if and only if $D_K$ is ``close to" a $D_{K'}$ which satisfies \eqref{e:distexact}. In the radial setting, where the only kernels which satisfy \eqref{e:distexact} are constants, measuring ``closeness" is relatively straightfoward:

\begin{thm}
Suppose that $K$ is radial and distance-standard. Then for any $d$-Ahlfors regular $\mu$ which is $d$-rectifiable in $\mathbb R^n$ and $\alpha > 0$ the non-tangential limits of $|\nabla D_{K, \mu,\alpha}|$ exist for $\mu$ a.e. in $E = \text{spt } \mu$ (for cones of every aperture) if and only if $K_\lambda  = K(\lambda \, \cdot) \ra c$ in $C^1_{loc}(\R^n \setminus \{0\})$ as $\lambda \da 0$ for some constant $c >0$.

Conversely, if for any $d$-Ahlfors regular measure $\mu$, the non-tangential limits of $|\nabla D_{K, \mu,\alpha}|$ exist for $\mu$ a.e. in $E = \text{spt } \mu$ (for every aperture), then $\mu$ must be $d$-rectifiable and $K$ must satisfy  $K_\lambda  = K(\lambda \, \cdot) \ra c$ in $C^1_{loc}(\R^n \setminus \{0\})$ as $\lambda \da 0$ for some constant $c >0$.\label{thm:rad_ntlintro}
\end{thm}

For kernels which are not rotationally invariant the analogue of Theorem \ref{thm:rad_ntlintro} is less clean due to the fact that $K(\lambda_j\, \cdot)$ could approach different kernels $K'$ satisfying \eqref{e:distexact} along different sequences $\lambda_j \downarrow 0$. In fact we construct such an example in Theorem \ref{t:nonradialnt}.

\vskip 0.08 in

Pertaining to question \eqref{intro:q2}, we work with a Dini-type condition:
 \begin{thm}
\sloppy Let $ 0 < d < n$ not necessarily an integer, and let $\alpha >0$. Suppose that $K \in C^3(\R^n \setminus \{0\})$ is radial, distance-standard, and $\nabla^3 K(x) |x|^3 \in L^\infty(\R^n)$. If $D_K$ satisfies the USFE, then $d$ is an integer and $\mu$ is $d$-uniformly rectifiable.

Conversely, if we assume in addition that $K$ is such that 
\begin{align}
\int_0^1 \left(t^{m}  \dfrac{d^m}{dt^m} \left(K(t) - K_0 \right) \right)^2  \; \dfrac{dt}{t} + \int_1^\infty \left(t^{m}  \dfrac{d^m}{dt^m} \left(K(t) - K_\infty \right) \right)^2  \; \dfrac{dt}{t} < \infty \label{cond:rad_usfe-intro}
\end{align}
for some constants $K_0, K_\infty >0$ and for $m=0,1,2$, then $D_K$ satisfies the USFE in $\Omega = \R^n \setminus \text{spt } \mu$ for any $d$-uniformly rectifiable measure $\mu$. \label{thm:rad_usfe}
\end{thm}
In the setting of uniform rectifiability, we are unable to prove that the integral condition \eqref{cond:rad_usfe-intro} on $K$ is sharp. This is due to our inability to quantify the non-existence of other radial kernels satisfying \eqref{e:distexact} in Theorem \ref{t:main}. However, we are able to show that $K$ must converge to appropriate constants near zero and infinity (see Theorem \ref{thm:necforusfe}). We also have results in the perturbative regime for non-radial $K$ (e.g. Lemma \ref{lemma:general_suff}). 

Let us conclude by outlining the remaining sections of the paper. In Section \ref{sec:class}, we study properties of distance-exact kernels (i.e. kernels satisfying \eqref{e:distexact}). In particular we show that for these kernels the answers to \eqref{intro:q1} and \eqref{intro:q2} are yes and we prove Theorem \ref{t:main}. The existence portion of Theorem \ref{t:main} is proven using a functional analytic approach followed by a smoothing argument. The construction of distance exact kernels which are far from being constant is proven using a scaling construction. Finally, we show that distance-exact radial kernels must be constants using Wiener-Tauberian theory.  

In Section \ref{sec:perturbative} we attempt to answer \eqref{intro:q1} and \eqref{intro:q2} using a perturbative analysis. In terms of non-tangential limits, \eqref{intro:q1}, we obtain a complete answer in Theorems \ref{thm:rect_ntl} and \ref{t:nonradialnt}. This is done using blowup arguments, specifically tangent measures. We also show by construction that our results are sharp without radial symmetry, see Theorem \ref{t:nonradialnt}. 

We then continue on to \eqref{intro:q2}, and develop sufficient conditions on $K$ to guarantee that $D_K$ satisfies the USFE for all planes, and from there, for all $d$-uniformly rectifiable measures. Our main result is Theorem \ref{T1}, which identifies a uniform condition on the growth of $D_K$ for the distance to satisfy the USFE outside of uniformly-rectifiable sets. This condition seems painful to check in practice but we give other simpler conditions on $K$ which guarantee that $D_K$ satisfies the uniform growth condition, for example, Lemma \ref{lemma:pert_char}, which shows that radial kernels satisfying the Dini-type condition \eqref{cond:rad_usfe} satisfy the uniform growth condition of Theorem \ref{T1}. The main difficulty in proving Theorem \ref{T1} (in contrast with  \cite{DEMMAGIC}) is that for general $K$, $F_{K, \HD^d|_E, \alpha}$ is not necessarily zero for $d$-planes $E$. 

Finally, we conclude Section \ref{sec:perturbative} by showing that under rather weak assumptions on $K$, the USFE with $D_{K, \mu, \alpha}$ implies that $\mu$ is uniformly rectifiable in Theorem \ref{thm:weakusfe_ur}. We also include a short appendix, proving that if $K$ is distance-exact (i.e. satisfies \eqref{e:distexact}), then the oscillations of $|\nabla D_K|$ characterize uniform rectifiability. 

\section{Distance-exact kernels}
\label{sec:class} 
Recall from \cite{DEMMAGIC} that a fundamental property of the regularized distance functions functions, $D_{1, \mathcal H^{d}|_E, \alpha}$, is that they are equal to (a multiple of) Euclidean distance when $E$ is affine. Let us generalize this notion for general kernels $K$:

\begin{defn}
If $K \in C(\R^n \setminus \{0\}) \cap L^\infty(\R^n)$ is a function such that for each $d$-plane $E$, there is a constant $c_E \in \R $ so that 
\begin{align}
D_{K, \HD^d|_ E, \alpha}(x) \equiv c_E \delta_E(x), \label{cond:distance_exact}
\end{align}
then we say that $K$ is $(d,\alpha)$-distance-exact.  If the constants $c_E \equiv 0$ for each $E$, we say that $K$ is $(d,\alpha)$-distance-orthogonal. Finally, if (\ref{cond:distance_exact}) holds for a single $d$-plane $E$, then we say that $K$ is $(d,\alpha)$-distance-exact for $E$.
\end{defn}

Analogously to \cite{DEMMAGIC} we start by showing that if a smooth enough $K$ is distance exact, then the oscillation of $|\nabla D_K|$ characterizes the regularity of $E$. In particular the following theorems hold: 

\begin{thm}\label{thm:dist_exact_ur_usfe}
Let $n,d \in \N$ with $1 \le d < n$, $\mu$ be a $d$-Ahlfors regular measure and let $\alpha >0$. If $K$ is distance-standard and $(d,\alpha$)-distance-exact, then $\mu$ is $d$-rectifiable if and only if for each $\eta \in (0,1)$, $\text{n.t.} \lim_{x \ra Q}^\eta |\nabla D_{K, \mu, \alpha}(x)|$ exists for $\mu$ almost all $Q\in \text{spt } \mu$  (See Definition~\ref{ntlimits} for the precise definition of non-tangential limits).
\end{thm}

\begin{thm}\label{thm:dist_exact_ntl_rect}
Let $n,d \in \N$ with $1 \le d < n$, $\mu$ be a $d$-Ahlfors regular measure and let $\alpha >0$.  If $K \in C^3(\R^n \setminus \{0\})$ is $(d, \alpha)$-distance-exact and $\nabla^3K(x) |x|^3 \in L^\infty(\R^n)$, then $D_{K, \mu, \alpha}$ satisfies the USFE if and only if $\mu$ is $d$-uniformly rectifiable.
\end{thm}

We present the proofs of Theorems \ref{thm:dist_exact_ur_usfe} and \ref{thm:dist_exact_ntl_rect} in the Appendix \ref{appendix}.  

In view of the above, to find kernels whose associated regularized distance characterizes geometric regularity, it suffices to understand distance-exact kernels. In what follows we first make some observations regarding distance-exact kernels with extra symmetries (i.e. radial or spherical invariance); this is the content of Section \ref{ss:distanceexactsymmetries}. Further explicit computations for zero-homogeneous kernels are left to Appendix \ref{appendix:homog_ker}, since the discussion is slightly tangential to the current one. The main result of this section (proven in Section \ref{ss:nonconstantdistexact}) is that there exist ``far from constant" distance-exact kernels which can be taken to be arbitrarily smooth (Theorem \ref{da_dist_exact_existence}). 

Before moving on we record the following useful Lemma:

\begin{lemma} \label{lemma:reduc_orth}
A function $K \in C(\R^n \setminus \{0\}) \cap L^\infty(\R^n)$ is $(d,\alpha)$-distance exact with constants $c_E \equiv c$ independent of the plane $E$ if and only if there some constant $\tilde{c}$ so that $K-\tilde{c}$ is $(d, \alpha)$-distance orthogonal. 
\end{lemma}

\begin{proof}
Recall that $K$ is $(d,\alpha)$-distance-exact with constants $c_E \equiv c$ independent of $E$ if and only if for every $E \in G(n,d)$ and every $x \not \in E$, we have
\begin{align*}
R_{K, E, \alpha} (x)  \equiv c^{-\alpha} \delta_E(x)^{-\alpha}.
\end{align*}
Given a $c$ there exists a $c_1\in \R$ such that if  $\tilde{K}(x) \equiv c_1$  then $R_{\tilde{K}, E,\alpha} = c^{-\alpha}\delta_E(x)^{-\alpha}$ for any affine $E$. The result follows from the linearity of $R_{K, E, \alpha}(x)$ in $K$ (for $E, x$ and $\alpha$ fixed).
\end{proof}

\subsection{Distance-exact kernels with additional symmetries}\label{ss:distanceexactsymmetries}
We briefly investigate distance-exact kernels with additional symmetry: either 0-homogeneity or rotational invariance. In particular, we first show that rotationally invariant (i.e. depending only on the radial variable) distance-exact kernels must be constant (this is Theorem \ref{thm:dist_exact_radial}). On the other hand we show that all $0$-homogeneous kernels whose associated distance characterizes good geometry must be distance exact. We leave to Appendix \ref{appendix:homog_ker} the existence and non-existence of non-constant distance-exact $0$-homogeneous kernels in various settings. 

\begin{thm}
Let $n,d \in \N$ with $1 \le d < n$ and let $\alpha >0$. Suppose that $K\in C(\R^n \setminus \{0\}) \cap L^\infty( \R^n)$ is $(d,\alpha)$-distance-exact and radial (i.e., $K(x) = \tilde{K}(|x|)$ for $\tilde{K} \in C(0, \infty) \cap L^\infty(0, \infty)$). Then $K$ is constant. \label{thm:dist_exact_radial}
\end{thm}
\begin{proof}
Since $K$ is invariant under rotations, the constants $c_E$ in the definition of distance-exactness are independent of $E \in G(n,d)$. In particular, Lemma \ref{lemma:reduc_orth} shows that $K-c$ is $(d,\alpha)$-distance-orthogonal for an appropriate constant $c$. In light of this, it suffices to prove that if $K$ is radial and $(d,\alpha)$-distance-orthogonal, then $K \equiv 0$. In what follows, we abuse notation and use $K$ both for the function defined on $\R^n$, and for the function $\tilde{K}$ defined on $(0,\infty)$ for which $K(x) = \tilde{K}(|x|)$.

Let $E$ be a $d$-plane (equipped with the Hausdorff measure), and let $x \not \in E$. Writing $R_K(x)$ as an integral in polar coordinates about the point $P_E(x) \in E$, one can compute
\begin{align*}
   R_K(x) & = c_d\int_{0}^\infty \dfrac{K(\sqrt{\delta_E(x)^2 + s^2})}{\bigl(\delta_E(x)^2+s^2\bigr)^{\frac{d + \alpha}2}}\, s^{d-1} \; ds \\
& = c_d \int_{\delta_E(x)}^{\infty} \dfrac{K(t)}{t^{d+\alpha - 1}} \bigl(t^2 - \delta_E(x)^2\bigr)^{\frac{d-2}2} \; ds.
\end{align*}
It follows that $K$ is $(d,\alpha)$-distance-orthogonal if and only if for each $r >0$, one has 
\begin{align}
    \int_{r}^\infty \dfrac{K(t)}{t^{d+\alpha - 1}}  \bigl(t^2 - r^2\bigr)^{\frac{d-2}2} \; dt = 0. \label{rdocond}
\end{align}
When $d = 2$, the fundamental theorem of calculus implies that $K \equiv 0$ for any $\alpha > 0$. When $d = 2j > 2$, one may differentiate the integral to obtain 
\begin{align*}
    \int_{r}^\infty  \dfrac{K(t)}{t^{d+\alpha - 1}} \bigl(t^2 - r^2\bigr)^{\frac{d-4}2} \; dt = 0.
\end{align*}
So if $K$ is $(d,\alpha)$-distance orthogonal then it is also $(d-2, \alpha+2)$-distance orthogonal. Repeating this process a total of $j-1$ times yields $K \equiv 0$. This completes the proof when $d$ is even.

For odd $d$ arguing as above reduces to the case when $d= 1$. To prove that the only $(1, \alpha)$-distance orthogonal radial kernel is trivial, we use Weiner-Tauberian theory. For each $r >0$ define the function
\begin{align*}
f_r(t) & = \chi_{(r, \infty)}(t)  \dfrac{1}{t^{\alpha}\sqrt{t^2-r^2}}.
\end{align*}
Let $W = \overline { \text{span} \{ f_r(t) \; : r > 0 \}  } \subset L^1(0, \infty)$, and note that $W$ is closed under dilations, since for $\lambda >0$
\begin{align*}
    f_r(\lambda t) & =  \chi_{(r, \infty)}(\lambda t)  \dfrac{1}{(\lambda t)^{\alpha}\sqrt{(\lambda t)^2 - r^2}} \\
& = \lambda^{-\alpha - 1} \chi_{(r/\lambda, \infty)}(t) \dfrac{ 1}{t^\alpha \sqrt{t^2 - (r/\lambda)^2}} \\
& = \lambda^{-\alpha - 1} f_{r/\lambda}(t) \in W.
\end{align*}
By definition, if $K$ is distance-orthogonal then $K \in W^\perp \subset L^\infty(0,\infty)$. Thus to show $K \equiv 0$ it suffices to show $W = L^1(0,\infty)$. To this end, we  consider the linear isomorphism $$T:L^1(0, \infty) \ra L^1(\R)$$ that maps $f(x)$ to $e^x f(e^x)$. Note that that $T (W) \subset L^1(\R)$ is a closed subspace of $L^1(\R)$ that is also closed under translations, to wit,
\begin{align*}
    (Tf)(x+a) & = e^a e^x f(e^a e^x)= e^a Tg(x)
\end{align*}
for $g(x) := f(e^a x) \in W$.

The Wiener Theorem (cf. Theorem 9.3 in \cite{RUDIN}) implies that $T(W) = L^1(\mathbb R)$ (and thus $K \equiv 0$) if
\begin{align*}
 Z(T(W)) & \equiv \bigcap\limits_{g \in T(W)} \{ s \in \R \; : \; \hat{g}(s) = 0 \} = \emptyset.
\end{align*}
This is a direct computation; we show that for each $s \in \R$, 
$$
\int_{\R} e^x f_1(e^x) e^{-2 \pi i x s} \; dx  = \int_{0}^\infty \dfrac{e^{x(1 - \alpha)}}{\sqrt{e^{2x}-1}} e^{-2 \pi i x s} \; dx \ne 0.
$$
For $s = 0$, this is obvious since $e^x f_r(e^x) \ge 0$. Furthermore the case $s > 0$ and $s <0$ are identical up to a change of sign, so we may assume $s > 0$. Observe that, $\dfrac{e^{x(1 - \alpha)}}{\sqrt{e^{2x}-1}}$ is decreasing on $(0,\infty)$, and thus 
\begin{align*}
\int_0^\infty  \dfrac{e^{x(1 - \alpha)}}{\sqrt{e^{2x}-1}} \sin(2 \pi x s) \; dx > 0,
\end{align*}
as desired.
\end{proof}

The next natural symmetry class to consider is homogeneous of degree $0$ kernels, i.e. $K$ such that $K(\lambda x) = K(x)$ for every $\lambda > 0$ and $x\in \mathbb R^n$. Our first observation is that for kernels that are homogeneous of degree zero, being distance-exact is a necessary requirement in order for $D_K$ to satisfy the USFE outside of each $E \in G(n,d)$:
\begin{thm}
Let $\alpha >0$ and let $K$ be a distance-standard, homogeneous of degree zero kernel with the property that for each $E \in G(n,d)$, $D_{K, \HD^d|_E, \alpha}$ satisfies the USFE. Then necessarily $K$ is $(d,\alpha)$-distance-exact. \label{thm:usfe_implies_dist_ex}
\end{thm}

\begin{proof}
Suppose that $E \in G(n,d)$ is fixed, and $\mu = \HD^d|_E$. Since $K$ is zero-homogeneous, it is easy to see that $R_K$ is homogeneous of degree $-\alpha$, and thus $D_K$ is homogeneous of degree $1$, and $F_K$ is homogeneous of degree zero. Moreover, since $D_K$ is translation invariant with respect to $E$ (because $E$ is a plane), we have that
\begin{align}
\int_{B(0,R)} F_K(x)^2 \delta_E(x)^{-n+d} \; dx  \ge c R^d \int_{E^\perp \cap B(0,cR)} F_K(x)^2 \delta_E(x)^{-n+d} \; d\HD^{n-d}(x), \label{eqn:fk_scale_inv}
\end{align}
where $E^\perp \in G(n, n-d)$ is the orthogonal complement to $E$ and $0 < c < 1$ is a dimensional constant so that $(B(0, cR) \cap E) + (B(0, cR)) \cap E^{\perp} \subset B(0, R)$. Since $D_{K, \HD^d|_E, \alpha}$ satisfies the USFE, we readily see that taking $R \ra \infty$ in (\ref{eqn:fk_scale_inv}) yields
\begin{align}
\int_{E^\perp} F_K(x)^2 \delta_E(x)^{-n+d}  d\HD^{n-d}(x) < \infty.  \label{eqn:fk_scale_inv_2}
\end{align}
We show this implies $F_K \equiv 0$. 

By the coarea formula with the Lipschitz function $\delta_E(x)$, we have that 
\begin{align*}
\int_{E^\perp \cap B(0,R)} F_K(x)^2 \delta_E(x)^{-n+d} \; & d\HD^{n-d}(x) \\
& = \int_0^R \left( \int_{E^\perp \cap \{  \delta_E(x) = t \} }F_K(x)^2  \delta_E(x)^{-n+d} \; d \HD^{n-d-1}(x) \right) \; dt \\ 
& = \int_0^R t^{-n+d}  \left( \int_{E^\perp  \cap \{  \delta_E(x) = t \} }F_K(x)^2  \; d \HD^{n-d-1}(x) \right) \; dt \\
& = \int_0^R t^{-1}  \left( \int_{E^\perp  \cap \{  \delta_E(x) = 1 \} }F_K(x)^2  \; d \HD^{n-d-1}(x) \right) \; dt \\
& =  \left( \int_{E^\perp  \cap \{  \delta_E(x) = 1 \} }F_K(x)^2  \; d \HD^{n-d-1}(x) \right) \int_0^R \dfrac{dt}{t}.
\end{align*}
Combining this with (\ref{eqn:fk_scale_inv_2}), we see that 
\begin{equation}\label{e:intfzero}
\int_{E^\perp  \cap \{  \delta_E(x) = 1 \} }F_K(x)^2  \; d \HD^{n-d-1}(x) = 0.
\end{equation}
Since $F_K$ is scale-invariant, and $F_K(x) = F_K(x + v)$ for any $v \in E$, we have that $F_K$ is constant on the set $\{\delta_E(x) = 1\}$, and if \eqref{e:intfzero} holds, $F_K \equiv 0$. From here it follows that $| \nabla D_K|$ is constant on $\mathbb R^d\backslash E$ and thus Theorem 3.1 in \cite{DEMMAGIC} implies that $D_K \equiv c \delta_E(x)$ for some $c = c_E>0$. This shows that $K$ is $(d,\alpha)$-distance-exact.
\end{proof}

In contrast to the radial case, we can construct many examples of $0$-homogeneous non-constant $(d,\alpha)$-distance-exact kernels (and even guarantee that the constant $c_E$ is independent of $E \in G(n, d)$). We leave such computations to Appendix \ref{appendix:homog_ker}, but want to draw attention to Corollaries \ref{rmk:dist_exact_codim_1} and \ref{rmk:dist_exact_codim_2}, which depending on the choice of $n, d, \alpha$, show that continuous examples of 0-homogeneous $(d,\alpha)$-distance exact kernels exist and or do not exist respectively.

\subsection{Existence of non-trivial distance-exact kernels}\label{ss:nonconstantdistexact}

The goal of this subsection is to prove the first part of our main Theorem \ref{t:main}:
\begin{thm}
For each choice of  $n, d, m\in \N,$ with $ 1 \le d < n$ and $\alpha >0$, there exists a non-constant $(d,\alpha)$-distance-exact kernel $K \in L^\infty(\R^n) \cap C^\infty(\R^n \setminus \{0\})$ so that $| \nabla^m K(x) |x|^m | \in L^\infty( \R^n)$. Moreover, $K$ can be constructed so that the constant $c_E$ in (\ref{cond:distance_exact}) is independent of $E$.  \label{da_dist_exact_existence}
\end{thm}

In light of Corollary \ref{rmk:dist_exact_codim_2} some of these kernels are not zero-homogeneous. Additionally, in contrast with Corollary \ref{rmk:dist_exact_codim_1}, we can guarantee the existence of said kernels  for any $\alpha > 0$ and in any co-dimension.

By Lemma \ref{lemma:reduc_orth} we can consider distance-orthogonal kernels. Our first step is to construct a non-zero, distance-orthogonal kernel using functional analytic methods which may or may not have the desired smoothness. In the following lemma, we denote by $C_0(\R^n \setminus \{0\})$ the closure, under the sup norm, of continuous functions with compact support in $\R^n\setminus \{0\}$. In particular if $f \in C_0(\R^n\setminus \{0\})$ then $f(x) \ra 0$ as $|x|$ goes to $\infty$ and $0$.

\begin{thm}
For each choice of $n \in \N$, $d \in \N$ with $1 \le d < n$, and $\alpha > 0$, there exists $K \in C_0(\R^{n} \setminus \{0\}) \cap L^\infty(\R^n)$ such that $K \not \equiv 0$, but $R_{K, E, \alpha} \equiv 0$ outside of $E$ for each $d$-plane $E \subset \R^n$. \label{dist_orth_existence}
\end{thm}
\begin{proof}
Take $X = C_0(\R^n \setminus \{0\}) \cap L^\infty(\R^n)$. As $X$ is a closed subset of $L^\infty(\R^n)$, it is a Banach space (when endowed with the supremum norm). Note that for $K \in C_0(\R^n \setminus \{0\}) \cap L^\infty(\R^n)$, we have that $R_{K, E, \alpha} \equiv 0$ for each $d$-plane $E \subset \R^n$ if and only if
\begin{align*}
\int_E \dfrac{K(z)}{|z|^{d+\alpha}} \; d\HD^d(z) =0 
\end{align*}
for each $E \in A(n,d)$ not containing the origin. This is because the affine change of variables $z = x-y$, which maps $E$ to the plane $x - E$, preserves $\HD^d$ measure. Hence, a kernel $K$ satisfies the conclusion of the Theorem if and only if $K$ is orthogonal to each measure of the form $d\mu(x) = |x|^{-d-\alpha} \; d\HD^d|_E(x)$ where $E \in A(n,d)$ does not contain the origin. For the sake of convenience, denote each such measure by $\mu_E$. Define $ M \subset X^*$ to be the weak-star closure of the subspace
\begin{align*}
\text{span} \{ \mu_E \; : \;   E \in A(n,d) \text{ with } 0 \not \in E \},
\end{align*}
where we view each measure $\mu_E$ as an element in $X^*$. Then $R_{K, E, \alpha} \equiv 0$ for each $E \in A(n,d)$ if and only if $K \in  {}^\bot M$, where
\begin{align*}
{}^\bot M \equiv \{f \in X \; : \; \Lambda( f) = 0 \text{ for all } \Lambda \in M\}.
\end{align*}

By the Hahn-Banach separation theorem, the existence of such a $K$ is equivalent to the existence of an $x \ne 0$ in $\mathbb R^n$ so that $\delta_x \not \in M$. Here, $\delta_x \in M$ is the functional $\delta_x(K') = K'(x)$. We will actually show something stronger, that $\delta_x \notin M$ for all $0 \neq x \in \mathbb R^n$. 

Let $x_0 \ne 0$ be given, and suppose for the sake of contradiction that $\delta_{x_0} \in M$. By definition there exist complex measures of the form $\nu_i = \sum_{j=1}^{m_i} a_j^i \mu_{E_j^i}$ where $a_j^i \in \C$ and $ 0 \not \in E_j^i \in A(n,d)$ are distinct such that $\nu_i \rightharpoonup \delta_{x_0}$. That is, for each $f \in X$, $\int f \; d\nu_i \ra f(x_0)$. Define $T_i: X \ra \C$ by $T_i(f) = \int f \; d\nu_i$. The $T_i$ are bounded linear functionals on $X$, and moreover,
\begin{align*}
     \sup_{i \in \N} |T_i(f)| < \infty
\end{align*}
since $T_i(f) \ra f(x_0)$ as $i \ra \infty$. By the Uniform Boundedness Principle, $\sup_i \|T_i\| = B < \infty$. One can check that $\|T_i\| = \sum_{j=1}^{m_i} |a_j^i| \;  \mu_{E_j^i}(\R^n)$, since the $E_j^i$ are distinct, and since distinct $E_j^i$ intersect in affine sets of zero $\HD^d$ measure.

Let $ \epsilon >0$ be given so that $ \epsilon < |x_0|$, and choose $\phi_{\epsilon} \in C_c(\R^n)$ so that $0 \le \phi_{\epsilon} \le 1$, $\phi_{\epsilon} \equiv 1$ on $B(x_0, \epsilon/2)$, and so that $\text{supp } \phi_{\epsilon} \subset B(x_0, \epsilon)$. Fix some $f \in X$ with $f(x_0) \ne 0$. A simple calculation yields that 
\begin{align*}
\delta_{x_0}(f) & = \delta_{x_0}(f \phi_\epsilon )   = \lim_{i \ra \infty} \int_{\R^n} f \phi_\epsilon \; d \nu_i.
\end{align*}
Note though that 
\begin{align*}
\left | \int_{\R^d} f \phi_\epsilon  \; d \nu_i \right| & \le \|f\|_\infty |\nu_i|(B(x_0, \epsilon)).
\end{align*}
If we can show that
\begin{equation}
|\nu_i|(B(x_0, \epsilon)) \da 0 \text{ uniformly in } i \text{ as } \epsilon \da 0, \label{eqn:reduc11}
\end{equation}
then we will have obtained the contradiction with $\delta_{x_0}(f) = 0$,  and conclude that $\delta_{x_0} \not \in M$.

For each $\epsilon$, we set
\begin{align*}
     C_\epsilon \equiv \sup_{0 \not \in E \in A(n,d)} \dfrac{\mu_E(B(x_0,  \epsilon))}{\mu_E(\R^n)} \le 1.
\end{align*}
Since $\mu_E(B(x_0, \epsilon)) \le C_\epsilon \mu_E(\R^n) < \infty$, for any $E \in A(n,d)$ not containing the origin, we have that 
\begin{align*}
|\nu_i|(B(x_0, \epsilon) )  \le \sum_{j=1}^{m_i} |a_j^i| \mu_{E_j^i}(B(x_0, \epsilon)) \le C_\epsilon \sum_{j=1}^{m_i} |a_j^i| \mu_{E_j^i}(\R^n)  = C_\epsilon \|T_i\|  \le C_\epsilon B.
\end{align*}
In particular, it suffices to show $C_\epsilon \da 0$ as $\epsilon \da 0$. 

Let $\epsilon \ll |x_0|$ and $E$ be an arbitrary $d$-affine plane not containing the origin. We consider two cases; first,  if $\mathrm{dist}(0, E) > 2 |x_0|$, then $B(x_0, \epsilon)\cap E = \emptyset$ and $\mu_E(B(x_0, \epsilon))/\mu_E(\mathbb R^n) = 0$. On the other hand if $\mathrm{dist}(0, E) < 2|x_0|$, then we have the lower bound $\mu_E(\mathbb R^n) \gtrsim \mathrm{dist}(0, E)^{-\alpha} \gtrsim |x_0|^{-\alpha}$. It is then easy to estimate $$\mu_E(B(x_0, \epsilon)) \lesssim \frac{\epsilon^d}{(|x_0|-\epsilon)^{d+\alpha}} \lesssim \frac{\epsilon^d}{|x_0|^{d+\alpha}},$$ where in the last inequality we used that $\epsilon \ll |x_0|$. Putting all this together we get that $$C_\epsilon \lesssim |x_0|^{\alpha} \frac{\epsilon^d}{|x_0|^{d+\alpha}} = \left(\frac{\epsilon}{|x_0|}\right)^d\stackrel{\epsilon \downarrow 0}{\rightarrow} 0.$$
\end{proof}

We cannot na\"ively adapt the above argument to guarantee that the kernel we obtain is smooth (in particular, distance-standard). This is because we do not have the crucial equality $\|T_i\| = \sum_j^i |a_j^i| \mu_{E_j^i}(\R^n)$ when we consider the norm in $(C^2(\R^n \setminus \{0\}))^*$. Instead, we smooth out the $K$ obtained above, first along each ray from the origin and then along each spherical shell.

\begin{lemma}
Let $n,d \in \N$ with $1 \le d < n$ and let $\alpha >0$. Let $\phi \in L^1(0,\infty)$, and suppose that $K \in C_0(\R^n \setminus \{0\}) \cap L^\infty(\R^n)$ is $(d,\alpha)$-distance-orthogonal. Then the kernel defined by 
\begin{align*}
\tilde{K}(x) \equiv \int_0^\infty K(tx) \phi(t) \; dt
\end{align*}
is also $(d,\alpha)$-distance-orthogonal. Moreover, if we assume that $\phi \in C^k(0,\infty)$ with
\begin{align*}
\int_0^\infty t^m |\phi^{(m)}(t)| \; dt < \infty
\end{align*}
for $m = 1,2, \dotsc, k$, then $f_x(t) =  \tilde{K}(xt)$ is $C^k(0,\infty)$ in $t$ for each $x \ne 0$, and 
\begin{align*}
\left \| \dfrac{d^m}{d t^m} (f_x(t))  t^m \right\|_{L^\infty(0, \infty)} \le M_m < \infty
\end{align*}
for $m=0, 1, \dotsc, k$. Here the $M_m$ are independent of $x$. \label{rad_smooth}
\end{lemma}

\begin{proof}
Observe that since $K$ is bounded, $\tilde{K}$ is also bounded, with $\|\tilde{K}\|_\infty \le \|K\|_\infty \|\phi\|_1$, and thus $R_{\tilde{K}}$ is well-defined. 

Note that since $K$ is $(d,\alpha)$-distance-orthogonal, the kernel $K_R(x) \equiv K(Rx)$ also is. Indeed, for each $d$-plane $E$ through the origin and each $x \not \in E$, we have 
\begin{align*}
R_{K_R, E}(x) & = \int_E \dfrac{K(R(x-y))}{|x-y|^{d+\alpha}} \; d\HD^d(y)= R^{\alpha} \int_E \dfrac{K(Rx - w)}{|Rx - w|^{d+\alpha}} \; d\HD^d(w) = R^\alpha R_{K, E}(Rx).
\end{align*}
As such, if $R_{K,E} \equiv 0$, then $R_{K_R, E} \equiv 0$ as well.

Now fix $E \in A(n,d)$. By Fubini's theorem, we compute $R_{\tilde{K}}(x)$ for $x \not \in E$:
\begin{align*}
     R_{\tilde{K}}(x) & = \int_{E} \dfrac{\tilde{K}(x-y)}{|x-y|^{d+\alpha}} \; d\HD^d(y) \\
& = \int_E \dfrac{1}{|x-y|^{d+\alpha}} \int_0^\infty K(t(x-y)) \phi(t)\; dt \; d\HD^d(y)  = \int_0^\infty \phi (t) R_{K_t, E} (x) \; dt 
  \equiv 0,
\end{align*}
so that $R_{\tilde{K}}$ is $(d,\alpha)$-distance-orthogonal. This proves the first claim.

Now let us suppose that $\phi \in C^k(0, \infty)$ as above. Then we remark that for $x \ne 0$,
\begin{align*}
\tilde{K}(x) & = \int_0^\infty K(xt) \phi(t) \; dt  = \dfrac{1}{|x|} \int_0^\infty K(xs/|x|) \phi(s/|x|) \; ds.
\end{align*}
It follows that 
\begin{align*}
\tilde{K}(\lambda x) & = \dfrac{1}{\lambda |x|} \int_0^\infty K(xs/|x|) \phi(s/(\lambda |x|)) \; ds.
\end{align*}
The right-hand side is differentiable in $\lambda$ with derivative
\begin{equation}\label{e:diffoff}
\begin{aligned}
\frac{d}{d\lambda} \tilde{K}(\lambda x) =\dfrac{-1 }{\lambda^2 |x|} &  \int_0^\infty K(xs/|x|) \phi(s/(\lambda |x|)) \; ds   - \dfrac{1}{\lambda^3 |x|^2} \int_0^\infty sK(xs/|x|) \phi'(s/(\lambda |x|)) \; ds   \\
& = \dfrac{-1}{\lambda} \left(  \int_0^\infty K(\lambda x t ) \phi(t) \; dt + \int_0^\infty K(\lambda x t) t \phi'(t) \; dt \right) \
\end{aligned}
\end{equation}
which is bounded in absolute value by $\lambda^{-1} \|K\|_\infty \left( \|\phi\|_1 + \|t \phi'(t)\|_1 \right)$. This proves the claim on the first derivative of $\lambda \ra K(\lambda x)$, and the arguments for the higher derivatives of this function follow in the same fashion.
\end{proof}

The following lemma uses the rotation invariance of $\delta_E$ to smooth out $K$ in the tangential directions while preserving the distance-orthogonality.

\begin{lemma}
Let $n,d \in \N$ with $1 \le d < n$ and let $\alpha >0$. Denote by $X$ the special orthogonal group $SO(n)$, and $\nu$ its Haar measure. Let $\phi \in L^1(X, \nu)$, and suppose that $K$ is bounded and $(d,\alpha)$-distance-orthogonal. Then the kernel defined by 
\begin{align*}
     \tilde{K}(x) & = \int_{X} K(Ax) \phi(A) \; d\nu(A)
\end{align*}
is $(d,\alpha)$-distance-orthogonal. Moreover, $\tilde{K}$ satisfies the two following smoothness conditions:

\begin{enumerate}[(I)]
\item If $\phi \in C^k(X)$, then for each $x \ne 0,$  we have that $A \ra K(Ax)$ is in $C^k(X)$ with uniformly bounded Lie derivatives. \label{tang_smooth_cond_1}
\item If for $x \ne 0$, the map $f_x(\lambda) =  K(\lambda x)$ is in $C^k(0, \infty)$, then so is the map $\tilde{f}_x(\lambda) = \tilde{K}(\lambda x)$. Moreover $|\tilde{f}^{(m)}_x(\lambda)| \lesssim \sup_{|y| = |x|}|f^{(m)}_y(\lambda)|$ for $m = 1,2, \dotsc, k$ with constant depending only on $n,d$ and $\|\phi\|_1$.
\end{enumerate} \label{tang_smooth}
\end{lemma}
\begin{proof}
With $\tilde{K}$ as above, we have that $\tilde{K}$ is bounded, since $\|\tilde{K}\|_\infty \le \|K\|_\infty \|\phi\|_1$, and thus $R_{\tilde{K}}$ is well-defined. 

The first part of the proof is quite similar to the argument of Lemma~\ref{rad_smooth}. Indeed, for each $d$-plane $E$ through the origin and each $x \not \in E$, we have 
\begin{align*}
R_{K_A, E}(x) & = \int_E \dfrac{K(A(x-y))}{|x-y|^{d+\alpha}} \; d\HD^d(y)   =  \int_{AE} \dfrac{K(Ax - w)}{|Ax - w|^{d+\alpha}} \; d\HD^d(w)  =  R_{K, AE}(Ax).
\end{align*}
Here we have used the fact that $A$ preserves Euclidean distance. Thus $K_A$ is $(d,\alpha)$-distance orthogonal whenever $K$ is. 

Now fix $E \in A(n,d)$. Invoking Fubini's theorem, we compute $R_{\tilde{K}}(x)$ for $x \not \in E$:
\begin{align*}
     R_{\tilde{K}, E}(x) & = \int_{E} \dfrac{\tilde{K}(x-y)}{|x-y|^{d+\alpha}} \; d\HD^d(y)  = \int_E \dfrac{1}{|x-y|^{d+\alpha}} \int_{X} K(A(x-y)) \phi(A) d\nu(A) \; d\HD^d(y) \\
& = \int_{X}  \phi (A) \int_E \dfrac{K(A(x-y))}{|x-y|^{d+\alpha}} \; d \HD^d(y) \; d\nu(A) = \int_{X} \phi (A) R_{K_A, E} (x) \; d\nu(A)  \equiv 0,
\end{align*}
so that $R_{\tilde{K}}$ is $(d,\alpha)$-distance-orthogonal. This proves the first claim.

Now let us suppose that $\phi \in C^k(X)$. Fix $x \in \R^n, x \ne 0$. Since $\nu$ is a Haar measure, we have
\begin{align*}
\tilde{K}(Bx) & = \int_{X} K(ABx) \phi(A) \; d \nu(A)   = \int_{X} K(Ax) \phi(AB^{-1}) \; d\nu(A).
\end{align*}
Since $\phi \in C^k(X)$, it is easy to see that the map $B \ra \int_X K(Ax) \phi(AB^{-1}) \; d\nu(A)$ also is, which gives the desired smoothness.

To verify that $\tilde{K}$ stays smooth in the radial direction (i.e. statement (II)), we compute
\begin{align*}
h^{-1} \left( \tilde{f}_x(\lambda + h) - \tilde{f}_x(\lambda) \right) & = \int_X h^{-1} \left( K((\lambda+h) Ax) - K(\lambda Ax)  \right) \phi(A) \; d\nu (A) \\
& = \int_X h^{-1} \left( f_{Ax}(\lambda +h) - f_{Ax}(\lambda) \right) \phi(A) \; d\nu(A),
\end{align*}
whence
\begin{align*}
\tilde{f}_x'(\lambda) & = \int_X f_{Ax}'(\lambda) \phi(A) d \nu(A).
\end{align*}
Since $|Ax| = |x|$ for $A \in X$, we also conclude the estimate on $|\tilde{f}_x'(\lambda)|$. The same argument is used to prove the statement for the higher order derivatives of $\tilde{f}$, and thus (II) is proved.
\end{proof}

Using these two smoothing lemmas we are ready prove Theorem \ref{da_dist_exact_existence}:

\begin{proof}[Proof of Theorem \ref{da_dist_exact_existence}]
In view of Lemma \ref{lemma:reduc_orth}, it suffices to show that one can construct a non-zero, smooth $(d,\alpha)$-distance-orthogonal kernel with the same smoothness.
This is essentially a combination of the two Lemmas above.

By Theorem \ref{dist_orth_existence}, we may choose a $(d,\alpha)$-distance-orthogonal kernel $K \in C_0(\R^n \setminus \{0\}) \cap L^\infty(\R^n)$ such that for some $x_0 \ne 0$, $K(x_0) \ne 0$. Choose $\phi \in C_c^\infty(0, \infty)$ so that 
$$\int_0^\infty K(tx_0) \phi(t) \;dt \ne 0.$$ By Lemma \ref{rad_smooth} (since $\phi$ is smooth with compact support in $(0,\infty)$) the kernel $K_1(x) \equiv \int_0^\infty K(tx)\phi(t) \; dt $ satisfies the following:
\begin{align}
\begin{cases}
&\text{for each  $x \ne$ 0}, t \ra K_1(tx) \in C^\infty(0, \infty),  \\
& \sup_{x \ne 0, t \in(0,\infty)} \left|t^m \dfrac{d^m}{dt^m} K_1(tx)\right| < \infty \text{ for } m \in \N, \\
& K_1(x_0) \ne 0, \\
& K_1 \text{ is } (d,\alpha)$-distance-orthogonal$.
\end{cases}
\end{align}

Denote by $X$ the special orthogonal group $SO(n)$, and $\nu$ its Haar measure. Next, choose $\psi \in C_c^\infty(X)$ so that $\int_{X} K_1(A x_0) \psi(A) \; d\nu(A) \ne 0$. Then by Lemma \ref{tang_smooth}, the kernel $K_2 \equiv \int_{X} K_1(Ax) \phi(A) \; d\nu(A)$ satisfies

\begin{align}
\begin{cases}
&\text{for each  $x\ne$ 0}, A \ra K_2(Ax) \in C^\infty(X),  \\
&\text{for each  $x\ne$ 0}, t \ra K_2(tx) \in C^\infty(0, \infty),  \\
& \sup_{x \ne 0, t \in(0,\infty)} \left|t^m \dfrac{d^m}{dt^m} K_2(tx)\right| < \infty  \text{ for } m \in \N, \\
& K_2(x_0) \ne 0, \\
& K_2 \text{ is } (d,\alpha)$-distance-orthogonal$.
\end{cases}
\end{align}

Since $K_2$ is smooth in the radial and tangential directions, we can conclude $K_2 \in C^\infty(\R^n \setminus \{0\})$. The bounds on the radial and tangential derivatives of $K_2$ coming from Lemmas \ref{rad_smooth} and \ref{tang_smooth} give us the required bounds on $|\nabla^m K_2(x)|\, |x|^m$.
\end{proof}

The next part of Theorem \ref{t:main} asks that the distance-exact kernels we construct be ``far" from constant at the origin. We do this by employing a scaling argument to show that the kernels we constructed above can be taken not to be ``close" to constant.

We begin with the observation that the kernels constructed as in the proof of Theorem \ref{da_dist_exact_existence} decay at $0$ and $\infty$ together with their properly normalized derivatives. 

\begin{lemma}\label{lem:decay}
For each choice of  $n \in \N, d\in \N$ with $ 1 \le d < n$ and $\alpha >0$, there exists a $(d,\alpha)$-distance-orthogonal kernel $0 \ne K \in L^\infty(\R^n) \cap C^\infty(\R^n \setminus \{0\})$ so that $|\nabla^m K(x) |x|^m| \in L^\infty( \R^n)$ for $m \ge 0$ and so that \begin{equation}\label{e:decay} \limsup_{|x|\rightarrow \infty} \sum_{m=0}^p |\nabla^m K(x) |x|^m| = 0 = \limsup_{|x|\rightarrow 0} \sum_{m=0}^p |\nabla^m K(x) |x|^m|.\end{equation}
for every $p \in \N$.
\end{lemma}

\begin{proof}
As in the proof of Theorem \ref{da_dist_exact_existence}, we start with a $K \in C_0(\mathbb R^n \backslash \{0\}) \cap L^\infty(\mathbb R^n)$ and construct $K_1(x) = \int_0^\infty K(tx)\phi(t)\, dt$ where $\phi \in C^\infty_c((0,\infty))$. Defining, as above, $f_x(\lambda) = K_1(\lambda x)$ we will show that for each $m \ge 0$ $$\limsup_{|x|\rightarrow \infty} \left| \left( \frac{d^m f_x(\lambda)}{d\lambda^m} \right)(1) \right| = 0 = \limsup_{|x|\rightarrow 0} \left| \left(\frac{d^m f_x(\lambda)}{d\lambda^m} \right) (1) \right|.$$ The result follows by continuing the construction as in the proof of Theorem \ref{da_dist_exact_existence} and the estimate above.

We do the case when $m = 1$, the others follow similarly. We recall from \eqref{e:diffoff} above that $$ \left|\lambda\frac{d}{d\lambda}f_x(\lambda)\right| \leq   \left| \int_0^\infty K(\lambda x t ) \phi(t) \; dt\right| + \left|\int_0^\infty K(\lambda x t) t \phi'(t) \; dt \right|.$$
Fixing $\lambda =1$ we notice that if $|x| \rightarrow 0, +\infty$ but $t\in \mathrm{spt} \phi$ then $\lambda t |x| \rightarrow 0, + \infty$ and indeed does so uniformly in all $t\in \mathrm{spt} \phi$. Since $K \in C_0(\mathbb R^n\backslash \{0\})$ this implies that $K(\lambda xt) \rightarrow 0$ and thus both integrals converge to zero in the limit, so we are done.
\end{proof}

We are now ready to address the second part of Theorem \ref{t:main}:

\begin{thm}\label{lem:dist_exact_nv}
For each choice of  $n \in \N, d\in \N$ with $ 1 \le d < n$ and $\alpha >0$ and $p \in \N$, there exists a $(d,\alpha)$-distance-orthogonal kernel $0 \ne K \in L^\infty(\R^n) \cap C^p(\R^n \setminus \{0\})$ so that $|\nabla^m K(x) |x|^m| \in L^\infty( \R^n)$ for $0 \le m \le p$ and, furthermore, \begin{align*}
\limsup_{|x| \ra 0} |K( x)| > 0.
\end{align*}
\end{thm}

We observe that adding such a kernel $K$ to any large enough constant gives the desired $(d,\alpha)$-distance exact kernel in Theorem \ref{t:main}.

\begin{proof}
Construct $K_0$ a non-zero $(d,\alpha)$-distance-orthogonal kernel, as in Theorem \ref{da_dist_exact_existence} which satisfies the estimate
\begin{align}
\sup_{x \ne 0} |\nabla^m K_0(x) | \, |x|^m = M_m < \infty \label{eqn:mod4}
\end{align}
for every $m \in \N \cup \{0\}$. Define $g_{K} : \R^n \setminus \{0\} \ra [0, \infty)$ by $g_{K}(x) = \sum_{\ell =0}^{p+1} |\nabla^\ell K(x)| \, |x|^\ell $. By Lemma \ref{lem:decay} we may assume $$\limsup_{|x| \downarrow 0} g_{K_0}(x) = 0 = \limsup_{|x|\uparrow \infty} g_{K_0}(x).$$ Furthermore, upon a harmless dilation and scalar multiplication (which both preserve distance orthogonality) we can assume that there is $x_0\in \mathbb S^{n-1}$ such that $K_0(x_0) = 1$. 

Our $K$ will be the limit of $K_j$ which are constructed iteratively. Let $\epsilon > 0$ be small. Choose a sequence $a_\ell$ decreasing monotonically to zero and $b_\ell$ increasing monotonically such that $a_\ell < 1 < b_\ell$ and $a_{\ell+1} \leq a_{\ell}^2/b_\ell$ and, finally, if $|x| \not\in (a_{\ell}, b_{\ell})$, then $g_{K_0}(x) \leq \epsilon 2^{-\ell}$. Define $$\tilde{K}_\ell(x) = K_0\left(\frac{b_\ell}{a_\ell}x\right),$$ and note, by scale invariance that 
\begin{equation}\label{e:gtildek}
g_{\tilde{K}_{\ell}}(x) \le \epsilon 2^{-\ell} \text{ for all } x \ne 0, |x| \not \in (a_{\ell}^2/b_{\ell}, a_{\ell}) \equiv I_{\ell}.
\end{equation}
We now define $K_j$ by \begin{equation}\label{e:KJdef}
K_j(x) = K_0 + \sum_{k=1}^{j} \tilde{K}_k(x).
\end{equation}
Note that $K_j$ is $(d,\alpha)$-distance orthogonal as it is the sum of distance orthogonal kernels. 

We want to show the following:
\begin{equation}\label{e:Kjcondition}
\begin{aligned}
g_{K_j}(x) \leq &\, 2\sup g_{K_0} + \epsilon,\\
K_j\left(\frac{a_\ell}{b_\ell}x_0\right) \geq &\, 1-\epsilon, \qquad \forall \ell \leq j.
\end{aligned}
\end{equation}
To give an upper bound on the estimate of $g_{K_j}$ we use the upper bound in \eqref{e:gtildek}, the triangle inequality, the disjointness of the $I_\ell$ and the scale invariance of the definition of $g$ to say
$$\begin{aligned} g_{K_j}(x) \leq&\, g_{K_0}(x) + \sum_{\ell = 1}^{j} g_{\tilde{K}_\ell}(x) \\
\leq& \,g_{K_0}(x) + \max_k g_{\tilde{K}_k}(x) + \sum_{\ell =1}^{j} \epsilon 2^{-\ell} \leq 2\sup g_{K_0} + \epsilon.
\end{aligned}$$

To get the lower bound on $K_j$ at the sequence of points $x_\ell := \frac{a_\ell}{b_\ell}x_0$ we observe that $\tilde{K}_k(x) \leq g_{\tilde{K}_k}(x)$ (and similarly for $K_0$)  and that if $k \neq \ell$ then $x_\ell \notin I_k$ to conclude that $$\begin{aligned} K_j\left(x_\ell\right) \geq& \,\tilde{K}_\ell(x_\ell) - \sum_{k \neq \ell} |\tilde{K}_k(x_\ell)| - |K_0(x_\ell)|\\
\geq& \, K_0(x_0) - \sum_{k\neq \ell} g_{\tilde{K}_k}(x_\ell) - g_{K_0}(x_\ell)
\geq K_0(x_0) -\sum_{k=1}^\infty \epsilon 2^{-k} = 1-\epsilon.\end{aligned}$$

Having proven the two conditions in \eqref{e:Kjcondition} we invoke Arzela-Ascoli and a standard diagonalization argument to say that from $K_j$ we may extract a subsequence $K_{j_\ell}$ and a limiting function $K \in C^{p}(\R^n \setminus \{0\})$ for which $K_{j_\ell} \ra K$ in $C^{p}_{loc}(\R^n \setminus \{0\})$, and 
\begin{align*}
\sum_{\ell = 0}^{p} |\nabla^\ell K(x)| \, |x|^\ell  \le 2 \sup g_{K_0}(y) + \epsilon.
\end{align*}
One may apply Lemma \ref{kernellemma} and the fact that each $K_{j_\ell}$ is $(d,\alpha)$-distance-orthogonal to deduce that $K$ is $(d,\alpha)$-distance-orthogonal. Moreover, \eqref{e:Kjcondition} implies that the limiting kernel $K$ has $K(x_k) \ge 1 - \epsilon$, for each $k$. Since $|x_k| \ra 0$, we have that $K$ is our desired kernel.
\end{proof}

\section{A perturbation theory for regularized distance kernels}
\label{sec:perturbative}
In this section we ask the perturbation question: if $K$ is ``close" to a distance-exact kernel does the oscillation of $|\nabla D_K|$ characterize good geometry 
and, vice versa, if the oscillation of $|\nabla D_K|$ characterizes good geometry must it be that $K$ is close to being distance exact?  Interestingly, using Theorem \ref{lem:dist_exact_nv} we show that just because $|\nabla D_K|$ characterizes good geometry does not mean it is a perturbation of a single distance exact kernel (cf. Theorem \ref{t:nonradialnt} below).

On the other hand, under the additional assumption of radial symmetry, which by Theorem \ref{thm:dist_exact_radial} simplifies the space of distance exact kernels, we are able to show that the oscillations of $|\nabla D_K|$ characterize the geometry of $E$ when $K$ is a perturbation of a constant; what we mean by perturbation depends on the context and we make it precise below. We also establish some weaker results in the absence of radial symmetry.

Finally,  as alluded to above, the direction ``good control on $|\nabla D_K|$ implies good geometry of $E$" holds for essentially all distance standard kernels $K$. This is because good control on $|\nabla D_K|$ actually implies that $K$ is close to being distance exact (cf. Corollaries \ref{c:structuredlimits} and \ref{c:almostconverse}).

\subsection{Non-tangential limits and rectifiability.}

Let us first address the question of rectifiability of $\mu$ in terms of non-tangential limits of $|\nabla D_{K, \mu}|$. We want to use compactness techniques so we first establish that if $K_i \ra K_\infty$ in the appropriate sense, then $R_{K_i} \ra R_{K_\infty}$ and $D_{K_i} \ra D_{K_\infty}$.
\begin{lemma}
Suppose that $\mu_i$ are a sequence of uniformly $d$-Ahlfors regular measures with supports $E_i$ such that $\mu_i \rightharpoonup \mu_\infty$. Let $E_\infty$ be the support of $\mu_\infty$. Suppose in addition that  $K_i \in C^k(\R^n \setminus \{0\})$ with $k \ge 0$,
\begin{align*}
    \sum_{j=0}^k  \sup_i \| \nabla^j K_i(x) |x|^j \|_\infty =: M  <\infty,
\end{align*}
and such that $K_i \ra K_\infty$ in $C^k_{loc}(R^n \setminus\{0\})$. It follows then that \begin{equation*}
    R_i(z) \equiv R_{K_i, \mu_i, \alpha} := \int_{E_i} \dfrac{K_i(z-w)}{|z-w|^{d+\alpha}} \; d\mu_i(w)
\end{equation*}
converges to
\begin{equation*}
    R_\infty(z) \equiv R_{K_\infty, \mu_\infty,\alpha} := \int_{E_\infty} \dfrac{K_\infty(z-w)}{|z-w|^{d+\alpha}} \; d\mu_\infty(w)
\end{equation*}
in $C^k_{loc}(\mathbb R^n \backslash E_\infty)$. The same holds true for $D_i = R_i^{-1/\alpha}$ and $D_\infty = R_\infty^{-1/\alpha}$.  \label{kernellemma}
\end{lemma}

\begin{proof}
     We prove only that $R_i \ra R_\infty$ in $C_{loc}(\R^n \setminus \{0\})$, since the argument for $\nabla^j R_i$ and the $D_i$ is essentially the same. Let $\epsilon >0$, and fix $A \subset \R^n\setminus E_\infty$ compact. By uniform Ahlfors regularity, choose $\rho \gg 1$ large enough so that 
     \begin{align*}
         \int_{\R^n \setminus B(0, \rho-1)} \dfrac{M}{|z-w|^{d+\alpha}} \; d\mu_i(w) < \epsilon
     \end{align*}
     holds for all $i \in \N$ sufficiently large, for $i = \infty$, and for all $z \in A$. The existence of such a $\rho$ follows from a standard argument using dyadic shells. Since $K_i \ra K_\infty$ uniformly on $\R^n \setminus \{0\}$, for $i$ sufficiently large we have
     \begin{equation*}
         \left | \int_{B(0,\rho)} \dfrac{K_i(z-w) - K_\infty(z-w)}{|z-w|^{d+\alpha}} d \mu_i(w) \right| < \epsilon, \qquad \forall z\in A. 
     \end{equation*}
It thus suffices to show that 
     \begin{align*}
         \left | \int_{B(0, \rho)} \dfrac{K_\infty(z-w)}{|z-w|^{d+\alpha}}  \; ( d\mu_i - d\mu_\infty )(w)\right| \le \epsilon, \qquad \forall z\in A,
     \end{align*}
     provided that $i$ is large enough. To obtain this, choose $\phi \in C_c^\infty(\R^n)$ with $\phi \equiv 1$ on $B(0, \rho-1)$, $0 \le \phi \le 1$, and supp $(\phi) \subset B(0,\rho)$. Then since $\mu_i \rightharpoonup \mu_\infty$, we have that 
     \begin{align*}
         \int_{B(0,\rho)} \dfrac{\phi(w) K_\infty(z-w)}{|z-w|^{d+\alpha}} \; d\mu_i(w) \ra \int_{B(0,\rho)} \dfrac{\phi(w) K_\infty(z-w)}{|z-w|^{d+\alpha}} \; d\mu_\infty(w)
     \end{align*}
     as $i \ra \infty$. The last terms that need to be estimated can be bound from above by
     \begin{align*}
         \int_{B(0,\rho) \setminus B(0, \rho-1)} \dfrac{M}{|z-w|^{d+\alpha}} \; (d \mu_i + d\mu_\infty)(w) < 2 \epsilon.
     \end{align*}
     This finishes the proof of convergence in $C_{loc}$. \end{proof}

Let us now preface our rectifiability results with some basic blow-up calculations. Assume that $E$ is a $d$-Ahlfors regular set equipped with the measure $\mu$. If $Q \in E, r_i > 0$ and $x_i \in \Omega$, let $X_i = (x_i - Q)/ r_i$. In addition, we consider the rescaled kernels $K_i(\cdot) = K(r_i \, \cdot)$ and the rescaled measures
\begin{align*}
    \mu_i(A) & \equiv \dfrac{\mu(r_i A + Q) }{r_i^d}.
\end{align*}
One easily checks that $\mu_i$ are uniformly $d$-Ahlfors regular (with constants only depending on the constants of $\mu$) with supports $E_i \equiv \dfrac{E - Q}{r_i}$, and that $K_i$ are distance-standard kernels with the same constants as $K$. Moreover, a simple change of variables yields
\begin{align*}
    D_{K_i, \mu_i}(X_i)^{-\alpha} & = \int_{E_i} \dfrac{K_i(X_i  -w)}{|X_i - w|^{d + \alpha}} \; d\mu_i(w) \\
    & = r_i^{d+ \alpha} \int_{E_i} \dfrac{K(x_i - (r_i w +Q))}{ |x_i - (r_i w +Q)|^{d+\alpha}} \; d\mu_i(w)   = r_i^{\alpha} D_{K, \mu}(x_i)^{-\alpha}.
\end{align*}
It follows then that 
\begin{align}
    |\nabla D_{K_i, \mu_i} (X_i)|  & = |\nabla D_{K, \mu}(x_i)|. \label{dkscale}
\end{align}

We recall some notation for non-tangential limits from Definition~\ref{ntlimits}.
With this language, we can now characterize rectifiability of $\mu$ in terms of non-tangential limits of $|\nabla D_{K, \mu}|$, provided that $K$ behaves like a constant near zero.

\begin{thm}
Let $n, d \in \N$ with $d < n$ and let $\alpha >0$. Let $K$ be a distance-standard, radial kernel. Then if $\mu$ is any $d$-Ahlfors regular measure with $d$-rectifiable support $E$, we have that the non-tangential limit $\ntlim_{x \ra Q}^\eta |\nabla D_K(x)|$ exists for every $\eta \in (0,1)$ at almost every $Q \in E$, if and only if $K(\lambda \,   \cdot) \ra c_\infty$ in $C^1_{loc}(\R^n \setminus \{0\})$ as $\lambda \da 0$ for some constant $c_\infty >0$ . \label{thm:rect_ntl}
\end{thm}
\begin{proof}
Let $0 <  \eta < 1$, $r_i \da 0$, and $x_i \in \Gamma_{1, \eta}(Q)$ such that $|x_i - Q| \downarrow 0$. Passing to a subsequence (which we relabel for convenience) we may assume that $K(r_i \, \cdot) \rightarrow K_\infty$ in $C^1_{\mathrm{loc}}(\mathbb R^n\backslash \{0\})$. Since $E$ is rectifiable, at almost every $Q\in E$ there is a unique tangent measure, which is flat (this measure may, of course, depend on $Q$, cf. \cite[Theorem 16.7]{Mattila}). As such we may assume that $\mu_i \rightharpoonup c\HD^d|_V$ (where $\mu_i$ is as above with respect to $r_i \downarrow 0$) and that the plane $V$ and constant $c$ are independent of the sequence $r_i\downarrow 0$ . 

Let $X_i = r_i^{-1}(x_i - Q)$ and since $\text{dist}(x_i, E) \ge \eta|x_i - Q|$, we have that $\text{dist}(X_i, \frac{E-Q_i}{r_i}) \ge \eta |X_i|$. If $|x_i - Q| \simeq r_i$ we have that (perhaps passing to a subsequence) $X_i \rightarrow X_\infty \in \mathbb R^d \backslash V$. 

By Lemma \ref{kernellemma}, the assumption that $K(r_i \, \cdot)\ra K_\infty$, and the previous calculations on the blowup of $D_{K_i, \mu_i}$, we have that, 
\begin{equation}\label{e:ntlimitexists}
\begin{aligned}
    |\nabla {D_{K_\infty, \mu_\infty}} (X_\infty)|  = \lim_{i \ra \infty} |\nabla {D_{K_i, \mu_i}}(X_i)|   = \lim_{i \ra \infty} |\nabla D_{K, \mu}(x_i)|.
\end{aligned}
\end{equation}
That the non-tangential limit of $|\nabla D_{K, \mu}|$ exists at $Q$ means that the limit in \eqref{e:ntlimitexists} is independent of  $\Gamma_{1,\eta}(Q) \ni x_i \rightarrow Q$. Fixing $r_i$ but adjusting $\eta$ and $x_i$  we can get every point $X_\infty \in \mathbb R^n \backslash V$. So the non-tangential limit exists if and only if  $K_\infty = \lim_{r_i\downarrow 0} K(r_i \, \cdot)$ is a kernel for which $|\nabla D_{K_\infty, \mu_\infty}|$ is constant outside $V$. By Corollary 3.2 in \cite{DEMMAGIC}, since $D_{K_\infty, \mu_\infty} \simeq \delta_{V}$, the only such functions are of the form $a \delta_V$ for some $a >0$. Thus we see that the non-tangential limits exists if and only if $K_\infty$ is distance-exact, for each $r_i \downarrow 0$ with a constant of exactness independent of the sequence $r_i \downarrow 0$ (recall that $\mu_\infty$ is independent of the sequence $r_i \downarrow 0$). By Theorem \ref{thm:dist_exact_radial}, the only distance-exact radial kernels are constants $c_\infty$ and the constant of exactness changes with $c_\infty$. So we conclude that the non-tangential limit exists if and only if $K(r \, \cdot) \rightarrow c_\infty$ for $r\downarrow 0$ and we are done.
\end{proof}

Without the assumption of radial symmetry we can only conclude that $D_{K_\infty}$ is distance-exact. Moreover, we have examples showing that it is possible to obtain a continuum of different $K_\infty$:

\begin{thm}\label{t:nonradialnt}
Let $n, d \in \N$ with $d < n$ and let $\alpha >0$. Let $K$ be a distance-standard kernel. Then if $\mu$ is any $d$-Ahlfors regular measure with $d$-rectifiable support $E$, we have that the non-tangential limit $\ntlim_{x \ra Q}^\eta |\nabla D_K(x)|$ exists for every $\eta \in (0,1)$ at almost every $Q \in E$, if and only if for every plane $V\in G(n,d)$ there exists a $c_V$ such that  $D_{K(\lambda \,   \cdot)} \ra c_V\delta_V$ in $C^1_{loc}(\R^n \setminus V)$ as $\lambda \da 0$.

Furthermore, we can construct a distance-standard kernel $K$ such that non-tangential limits of $|\nabla D_K|$ exist at almost-every point of any plane $V \in G(n,d)$ but such that the limit of $K(r_i \, \cdot)$ as $r_i\downarrow 0$ depends on the sequence $r_i$.
\end{thm}

\begin{proof}
The first part of the Theorem follows by arguing exactly as in Theorem \ref{thm:rect_ntl}. Indeed, in the proof of Theorem \ref{thm:rect_ntl} we only invoke radial symmetry once we have concluded that $|\nabla D_{K_\infty}|$ must be a constant outside of each plane and that constant must not depend on the sequence $r_i \downarrow 0$. 

Let $\tilde{K}$ be a smooth distance-orthogonal kernel satisfying the usual estimate $$\nabla^m \tilde{K}(x) |x|^m \in L^\infty(\R^n)$$ for $m \in \N$ which is not identically equal to zero but such that $\lim_{|x| \rightarrow 0} \tilde{K} = 0 = \lim_{|x| \rightarrow \infty} \tilde{K}$ (guaranteed to exist by Lemma \ref{lem:decay}). Let $1 = a_0 > a_1 > a_2 > \ldots > a_n \rightarrow 0$ such that $a_{i+1}/a_i \downarrow 0$ monotonically. Let $\phi_i\in C_c^\infty(a_{i+1}/2, 2a_i)$ be such that $0 \leq \phi_i \leq 1$, $\sum \phi_i \equiv 1$ on $(0,1)$ and so that $\phi_i \equiv 1$ on $(2a_{i+1}, a_i/2)$. Despite these constraints, a scaling arguments shows that we can still guarantee that $\||x|^m \nabla^m \phi_i\|_\infty \leq C_m$ for some $C_m > 0$ independent of $i$. Define for $M > \| \tilde{K} \|_\infty$ $$K(x) := M + \sum_{i=1}^\infty \phi_i(|x|)\tilde{K}\left(\frac{x}{\sqrt{a_ia_{i+1}}}\right).$$

First, we note that $K$ is distance standard, by the fact that $\tilde{K}$ is smooth, $M$ is large and the estimates on the derivatives of the $\phi_i$ and $\tilde{K}$. We want to show that for every $\lambda_i\downarrow 0$ there exists a $\lambda_{i_j}$ such that $K(\lambda_{i_j} \, \cdot) \rightarrow K_\infty$ where $K_\infty$ is distance exact. Note by passing to a subsequence and relabeling, we may assume that $a_{i+1} \leq \lambda_i \leq a_i$. 

We have two cases:

\noindent {\bf Case 1:} Here we assume $$0 < \liminf_i \frac{\lambda_{i}}{\sqrt{a_{i}a_{i + 1}}} \leq \limsup \frac{\lambda_{i}}{\sqrt{a_{i}a_{i+ 1}}} < \infty.$$ In this case, passing further to a subsequence we may assume that $\lim_{i}\lambda_i/\sqrt{a_ia_{i+1}} = \alpha \in (0, \infty)$. Fix $\mathcal K \subset \subset \mathbb R^n\setminus \{0\}$. For all $i$ large enough  and all $x\in \mathcal K$ we have that $\lambda_i |x| \in (a_{j+1}, a_j)$ if any only if $i = j$. On the other hand if $i$ is large enough (depending only on $\mathcal K$), we have that $\lambda_i |x| \in (2a_{i+1}, a_i/2)$ for all $x\in \mathcal K$ (this is because $a_{i+1}/\sqrt{a_ia_{i+1}} \rightarrow 0$ and $a_{i}/\sqrt{a_i a_{i+1}} \rightarrow +\infty$). Thus 
$$K(\lambda_i \,  \cdot) -(M+ \phi_i(\lambda_ix)\tilde{K}(\lambda_i x/\sqrt{a_ia_{i+1}})) \rightarrow 0 \mbox{ in $C^1(\mathcal K)$}$$  and $$ M+ \phi_i(\lambda_ix)\tilde{K}(\lambda_i x/\sqrt{a_ia_{i+1}}) \rightarrow M+ \tilde{K}(\alpha x) \mbox{ in $C^1(\mathcal K)$.}$$ We know $M+ \tilde{K}(\alpha x)$ is distance exact so we are done in this case. 
\medskip

\noindent {\bf Case 2:} We assume that either the $\liminf = 0$ or $\limsup = +\infty$ above. The arguments for the two cases are similar so let us just do the case when $\liminf_i \frac{\lambda_{i}}{\sqrt{a_{i}a_{i + 1}}} = 0$. Relabeling we may assume that $\lim_i \frac{\lambda_{i}}{\sqrt{a_{i}a_{i + 1}}} = 0$. However, we still have $\lambda_i \in (a_{i+1}, a_i)$. Let $\mathcal K \subset \mathbb R^n \backslash \{0\}$ be a compact set and observe for any $x\in \mathcal K$ we have $$\lim_{i \rightarrow \infty} \sup_j |\tilde{K}\left(\lambda_i x/\sqrt{a_ja_{j+1}}\right)| = 0.$$ Indeed this follows from the fact that $\tilde{K}$ goes to zero at zero and infinity, that $\frac{\lambda_{i}}{\sqrt{a_{i}a_{i + 1}}} \rightarrow 0$ and that if $i \neq j$ but $i$ is large enough we have $\frac{\lambda_{i}}{\sqrt{a_{j}a_{j+ 1}}} $ either blows up or goes to zero. 

As such $K(\lambda_i \, \cdot) \rightarrow M$ in $C^1(\mathcal K)$ in this case. Of course, constant kernels are distance exact. 

Finally, we see that by letting $\lambda_i = \sqrt{a_ia_{i+1}}$ or $\lambda_i = a_i$ we get that $K(\lambda_i \, \cdot)$ converges in $C^1_{\mathrm{loc}}(\mathbb R^n \backslash \{0\})$ to $M + \tilde{K}$ or $M$ respectively. Since these two kernels are different we are done. 
\end{proof}

We end by observing that even without any symmetry assumptions on $K$ the existence of non-tangential limits of $|\nabla D_{K, \mu}|$ implies the rectifiability of $\mu$. 

\begin{thm}
Let $0 < d < n$ not necessarily an integer and let $\alpha >0$.  Let $K$ be a distance-standard kernel. Suppose that $\mu$ is a $d$-Ahlfors regular measure with support $E$ such that $\mu$-almost everywhere, $\ntlim_{x \ra Q}^\eta |\nabla D_K(x)|$ exists each $\eta \in (0,1)$. Then $d$ is an integer and $\mu$ is $d$-rectifiable.  \label{thm:ntl_rect}
\end{thm}
\begin{proof}
  We show that $\mu$ is rectifiable by showing that almost everywhere in its support, all of its tangent measures are flat. Since $\mu$ has positive lower density and finite upper density (by Ahlfors regularity), Theorem 16.5 in \cite{Mattila} shows that these are equivalent conditions.
  
  Let $Q \in E$ be a point of $E$ so that $\ntlim_{x \ra Q}^\eta |\nabla D_{K, \mu}(x)|$ exists for each $\eta \in (0,1)$. Let $r_i \da 0$, and define $K_i, \mu_i, E_i$ as above. Up to a subsequence, we may as well assume the convergence of $\mu_i \rightharpoonup \mu_\infty$ and $E_i \rightarrow E_\infty$. Moreover, since $K$ is distance-standard, we may also assume that up to a subsequence, $K_i \ra K_\infty$ in $C^1_{loc}(\R^n \setminus \{0\})$ for some kernel $K_\infty$ that is strictly positive and satisfies $\nabla K_\infty(x) |x| \in L^\infty(\R^n)$. By Lemma \ref{kernellemma}, we may assume up to a subsequence that $D_{K_i, \mu_i}$ converges uniformly on compact subsets of $\Omega_\infty = \R^n \setminus E_\infty$ to $D_{K_\infty, \mu_\infty}$.
  
  Let $Z \in E_\infty$ and observe that if $\eta_Z := \text{dist}(Z, E_\infty) / (2|Z|) \in (0,1)$, then the points $x_i = Q + r_i Z$ satisfy $x_i \in \Omega$ for all $i$ sufficiently large with
  \begin{align*}
      \text{dist}(x_i, E) & = r_i \text{dist}(Z, E_i)  \ge (r_i/2) \text{dist}(Z, E_\infty)  = \eta_Z |Z|r_i   = \eta_Z |x_i - Q|.
  \end{align*}
  In particular then, for all $i$ sufficiently large we have $x_i \in \Gamma_{|Z| r_i, \eta_Z}(Q)$, and thus we have by assumption that 
  \begin{align*}
      \ntlim\nolimits_{x \ra Q}^{\eta_Z} |\nabla D_{K, \mu} (x_i)| & = \lim_{i \ra \infty} |\nabla D_{K_i, \mu_i}(Z)|   = |\nabla D_{K_\infty, \mu_\infty}(Z)|.
  \end{align*}
  Since $\ntlim_{x \ra Q}^{\eta_Z} |\nabla D_{K, \mu} (x_i)|$ is independent of $Z$, we have that $|\nabla D_{K_\infty, \mu_\infty}|$ is constant on $\Omega_\infty$. This constant cannot be zero, since $D_{K_i, \mu_i}$ is comparable to $\delta_{E_i}$ (with constants independent of $i$), and thus $D_{K_\infty, \mu_\infty}$ is comparable to $\delta_{E_\infty}$. By \cite[Corollary 3.2]{DEMMAGIC}, we have that $d \in \N$, $E_\infty$ is a $d$-plane, and $\mu_\infty$ is a constant multiple of $\HD^d|_{E_\infty}$. This shows that at this point $Q$, all tangent measures are flat, and thus the claim is proved.
\end{proof}

A curious corollary of the above results for any $d$-Ahlfors regular measure $\mu$, is that the existence of non-tangential limits $|\nabla D_K|$ for $\mu$-almost every $Q$ implies structure on $K$: 

\begin{cor}\label{c:structuredlimits}
Let $0 < d < n$ not necessarily an integer and let $\alpha >0$.  Let $K$ be a distance-standard kernel. Suppose that $\mu$ is a $d$-Ahlfors regular measure with support $E$ such that $\mu$-almost everywhere, $\ntlim_{x \ra Q}^\eta |\nabla D_K(x)|$ exists each $\eta \in (0,1)$. Then for every plane $V\in G(d,n)$ there exists a $c_V$ such that  $D_{K(\lambda  \, \cdot)} \ra c_V\delta_V$ in $C^1_{loc}(\R^n \setminus V)$ as $\lambda \da 0$.
\end{cor}
\subsection{The USFE for non-exact kernels}

Our aim in this section is to develop necessary and sufficient conditions on a distance-standard kernel $K$ so that $D_K$ satisfies the USFE outside of all $d$-uniformly rectifiable measures. The key idea is to measure how close $D_K$ is to behaving like the Euclidean distance outside of affine spaces, in a way which is uniform over scales and affine spaces. This is the purpose of the $\gamma$ function defined below in Definition \ref{defn:gamma}.

Before we define the $\gamma$ function, let us remember that for any distance-standard kernel $K$, outside of any affine space $E \in G(n,d)$ the functions $R_K$ (and thus $D_K, F_K$) are invariant in directions parallel to $E$. In particular, assume that $0 \in V \in G(n,d)$ and that $Q(0,r)$ is a cube of side length $r$ centered at $0$. Then $$\int_{Q(0,r)} F^2_{K, E}(x)\, \delta_E(x)^{-n+d}\, dx = r^d \int_{Q(0,r)\cap V^\perp} F^2_{K, E}(y)\, \delta_E(y)^{-n+d}\, d\mathcal H^{n-d}(y),$$ where we abused notation and identified a function $f:\mathbb R^n\rightarrow \mathbb R$ with its restriction  $f:\mathbb R^{n-d}\cong V^\perp\rightarrow \mathbb R$ by $f(y) := f(0+y)$. We will continue this abuse of notation throughout the section, hoping that it does not cause too much confusion.

\begin{defn}\label{defn:gamma}
Whenever $K$ is distance-standard, $\lambda > 1$ and $\alpha >0$, we write for $r >0$, $W_\lambda(r) = B(0, r) \setminus B(0, r/\lambda)$ and 
\begin{align}
\gamma_{K, \lambda, \alpha}(r)^2   \equiv \sup_{V\in G(d, n)}\inf_{c \in \R} \sum_{m=0}^2 \,\int \limits_{V^\perp \cap W_\lambda(r) } |\delta_{V}(z)^{\alpha + m} \nabla^m ( R_{{K}, V, \alpha} (z) - &  c \delta_{V}(z)^{-\alpha})|^2 \nonumber \\
& \times \delta_{V}(z)^{-n+d} \; d\HD^{n-d}(z). \label{eqn:gamma}
\end{align}
We say that $K$ is uniformly good for distances (with exponent $\alpha >0)$ if 
\begin{align}
\int_0^\infty \gamma_{K, \lambda, \alpha}(r)^2 \; \dfrac{dr}{r} < \infty \label{eqn:good_for_dist}
\end{align}
for some $\lambda >1$. 
\end{defn}

\begin{rmk}\label{rmk:good_for_dist}
It is straight-forward to verify that if $1 < \lambda_1 < \lambda_2$, then the estimate 
\begin{align*}
\int_0^\infty \gamma_{K, \lambda_1, \alpha}(r)^2 \; \dfrac{dr}{r} & \le   \int_0^\infty \gamma_{K, \lambda_2, \alpha}(r)^2 \; \dfrac{dr}{r} \le C_{\lambda_1, \lambda_2, \alpha} \int_0^\infty \gamma_{K, \lambda_1, \alpha}(r)^2 \; \dfrac{dr}{r}
\end{align*}
holds. Hence $K$ is uniformly good for distances if and only if 
\begin{align*}
\int_0^\infty \gamma_{K, 2, \alpha}(r)^2 \; \dfrac{dr}{r} < \infty. 
\end{align*}
\end{rmk}

The key estimate is to bound the integral of $F$ on $W_\lambda(r)$ by the $\gamma$ functional. 

\begin{lemma} \label{lemma:usfeflat}
      Suppose that $K$ is a distance-standard kernel.  There is a constant $C$ depending only on $n, d, \alpha$ and the distance-standard constant for $K$ so that the following estimate holds for any $\lambda >1$ and any $V\in G(n,d)$, 
\begin{align}
\int_{V^\perp \cap W_{\lambda}(r)} F_{K, V, \alpha}(z)^2  \delta_{V}(z)^{-n+d} \; d\HD^{n-d}(z) \le C \gamma_{K, \lambda, \alpha}(r)^2.  \label{eqn:mod7}
\end{align}
\end{lemma}

\begin{proof}
Since \eqref{eqn:gamma} is a supremum over the Grassmanian and since the rotation of a distance standard kernel is distance standard (with the same constants), \eqref{eqn:mod7} is rotation invariant. So we can assume that $V = E \equiv \mathbb R^d \subset \mathbb R^n$. 

Denote by $R_{c}(x), D_{c}(x), F_{c}(x)$ the corresponding functions with constant kernels that give $R_{c}(x) = c \delta_E(x)^{-\alpha}$ where $c \in \R$. Recall that since $E$ is a $d$-plane, we have that $F_{c}(x) \equiv 0$. We first calculate for $j > d$, 
 \begin{align*}
    \partial_j D_{K}(x) & =  \left(\dfrac{-1}{\alpha} \right) R_{K}(x)^{-1/\alpha- 1} \partial_j R_{K}(x),\\
    \left| \nabla D_{K}(x) \right|^2 & = \left( \dfrac{1}{\alpha^2} \right) R_{K}(x)^{-2/\alpha - 2} \left ( \sum_{j > d} |\partial_j R_{K}(x)|^2 \right),\\
    \partial_j \left | \nabla D_{K}(x) \right|^2 & = \left( \dfrac{1}{\alpha^2} \right) ( -2 / \alpha - 2) R_{K}(x)^{-2/\alpha - 3} \partial_j R_{K}(x) \left ( \sum_{i > d} |\partial_i R_{K}(x)|^2 \right)  \\
    &\quad + \dfrac{2}{\alpha^2} R_{K}(x)^{-2/\alpha - 2} \left( \sum_{i > d} \partial_i R_{K}(x) \partial_j \partial_i R_{K}(x) \right).
\end{align*}
Recall that $K$ is distance standard so for $m=0,1,2$,
\begin{align}
 \|\nabla^m K(x) |x| \|_\infty < \infty \label{eqn:mod6}
\end{align}
and straightforward estimates give
\begin{align}
    \delta_{E}(x)^{- \alpha} \lesssim& \,R_{K}(x) \lesssim \delta_{E}(x)^{- \alpha},\label{crude0}\\
        \left| \partial_j R_{K}(x) \right| \lesssim& \,\delta_{E}(x)^{-1  - \alpha}, \label{crude1} \\
    \left| \partial_j \partial_i R_{K}(x) \right| \lesssim& \,\delta_{E}(x)^{-2 - \alpha}. \label{crude2}
\end{align}

Putting \eqref{crude0} together with the computation of $|\partial_j |\nabla D_{K}|^2|$ we get
\begin{align}
     \delta_{E}(x)^{2 + 3 \alpha} B_{K}^j(x) \lesssim  \left| \partial_j \left | \nabla D_{K}(x) \right|^2 \right| & \lesssim  \delta_{E}(x)^{2 + 3 \alpha} B_{K}^j(x) \label{simplification}
\end{align}
where
\begin{align*}
    B_{K}^j(x) & = \left | (-1/\alpha - 1) \partial_j R_{K}(x) \left ( \sum_{i > d} |\partial_i R_{K}(x)|^2 \right)  +  R_{K}(x) \left( \sum_{i > d} \partial_i R_{K}(x) \partial_j \partial_i R_{K}(x) \right) \right |
\end{align*}
Remark that $B_{c}^j(x) \equiv 0$ necessarily, since otherwise $F_{c}(x) \ne 0$. We are now in the position to establish \eqref{eqn:mod7}.

We choose a constant $c_0 \in \R$ so that 
\begin{align}
\sum_{m=0}^2 \int_{E^\perp \cap W_\lambda(r)}   |\delta_{E}(z)^{\alpha + m} \nabla^m ( R_{K} (z) - c_0 \delta_{E}(z)^{-\alpha})|^2 \delta_{E}(z)^{-n+d} \; d\HD^{n-d}(z)  & \le 2 \gamma_{K, \lambda}(r)^2. \label{eqn:mod8}
\end{align}
 Recalling that $B^j_{c_0} \equiv 0$, we have 
\begin{align*}
    |B_{K}^j(x)| & = |B_{K}^j(x) - B_{c_0}^j(x)|,
\end{align*}
and each of the terms in the latter absolute value can bounded from above by 
\begin{align*}
   C \sum_{m = 0}^2 |\nabla^m (R_{K}(x) - c_0 \delta_{E}(x)^{-\alpha})| \delta_{E}(x)^{-2\alpha -3 + m},
\end{align*}
where $C$ depends on $n, d, \alpha$ and the distance-standard constant of $K$. Let us show part of this computation, as each term can be handled essentially the same way. To estimate the first terms appearing in $B_{K}^j - B_{c_0}^j$, we use \eqref{crude1}  to obtain
\begin{multline*}
    |\partial_j R_{K}(x) \partial_i R_{K}(x)^2 - \partial_j R_{c_0}(x) \partial_i R_{c_0}(x)^2 |  \le | \partial_i R_{K}(x)|^2 |\partial_j R_{K}(x) - \partial_j R_{c_0}(x)| \\
     + |\partial_j R_{c_0}(x) \partial_i R_{K}(x)| |\partial_i R_{K}(x) - \partial_i R_{c_0}(x) |  + |\partial_j R_{c_0}(x)^2 | |\partial_i R_{K}(x) - \partial_i R_{c_0}(x)| \\
     \lesssim \delta_{E}(x)^{-2-2\alpha} (|\partial_j R_{K}(x) - \partial_j R_{c_0}(x)|   
     + |\partial_i R_{K}(x) - \partial_i R_{c_0}(x) | + |\partial_i R_{K}(x) - \partial_i R_{c_0}(x)| ) \\
     \lesssim |\nabla(R_{K}(x) - c_0 \delta_{E}(x)^{-\alpha})| \delta_{E}(x)^{-2\alpha - 2}.
\end{multline*}
 Combining this with \eqref{simplification} we obtain the pointwise estimate
\begin{align*}
F_{K}(x) \le C  \sum_{m=0}^2 |\nabla^m (R_{K}(x) - c_0 \delta_{E}(x)^{-\alpha})| \delta_{E}(x)^{\alpha+m}.
\end{align*}
From here we readily see that 
\begin{align*}
    & \int_{E^\perp \cap W_\lambda(r)} F_{K}(z)^2 \delta_{E}(z)^{-n + d } \; d\HD^{n-d}(z)    \\
     & \quad\le  C \sum_{m=0}^2 \int_{E^\perp \cap W_\lambda(r)} |\delta_{E}(z)^{\alpha +m} \nabla^m(R_{K}(z) - c_0\delta_{E}(z)^{-\alpha})|^2 \delta_{E}(z)^{-n + d} \;  d\HD^{n-d}(z) \\
& \quad\le C \gamma_{K, \lambda, \alpha}(r)^2,
\end{align*}
by the choice of $c_0$, completing the proof of \eqref{eqn:mod7}.
\end{proof}

One can pass from estimates out of flat sets to estimates outside of uniformly rectifiable sets following \cite[Theorem 2.1]{DEMMAGIC}. However, there will be extra, complicating terms in the analysis, caused by the fact that $F_K$ may not be identically equal to zero outside of flat sets. We estimate those extra terms using the $\gamma$ functional and Lemma \ref{lemma:usfeflat}. 

\begin{thm}
Let $n, d \in \N$ with $d < n$, and let $\beta >0$. Suppose that $K$ is a distance-standard kernel that is uniformly good for distances, i.e., (\ref{eqn:good_for_dist}) holds for some $\lambda >0$ with exponent $\beta$. If in addition
\begin{align*}
    \nabla^3K(x) |x|^3 \in L^\infty(\R^n),
\end{align*}
then $D_{K, \mu, \beta}$ satisfies the USFE for each $d$-uniformly rectifiable measure $\mu$. \label{T1}
\end{thm}
\begin{proof}
Fix $\beta >0$ (which we shall omit in notation of $F_K$ and $D_K$), and fix $\mu$ some $d$-uniformly-rectifiable measure with support $E$. Let $\mathcal{F} = \mathcal{F}_d$ be the set of flat measures, and let $\mathcal{D}_{x,r}(\mu, \nu)$ denote the Wasserstein distance
\begin{align*}
\mathcal{D}_{x, r}(\mu, \nu) = r^{-d-1} \sup_{f \in \Lambda(x,r)} \left|  \int_{B(x,r)} f ( d\mu - d\nu) \right|
\end{align*}
where $\Lambda(x,r)$ is the set of functions $f$ that are $1$-Lipschitz on $\R^n$ and vanish on $\R^n \setminus B(x,r)$. With this distance, assign the definition
\begin{align*}
    \alpha(x,r) = \inf_{\nu \in \mathcal{F}} \mathcal{D}_{x,r}(\mu, \nu)
\end{align*}
where $x \in \R^n$, and $r>0$ is such that $B(x,r) \cap E \ne \emptyset$. These so called ``$\alpha$-numbers'' are useful in the context of quantitative rectifiability, since any uniformly rectifiable set $E$ satisfies Carleson measure estimates on $\alpha(x,r)^2$. In particular, since $\mu$ is $d$-uniformly rectifiable we have that 
\begin{align}
\int_{B(Q,R) \cap E} \int_0^R \alpha(y,r)^2 \dfrac{d \mu(y) \; dr}{r} \le C R^d, \label{alpha_cm}
\end{align}
as in Lemma 5.9 in \cite{DFM19} (see also \cite{TOLSA09} where the $\alpha$ numbers are introduced).

The key estimate we wish to show is the following: for $x \in \Omega = E \setminus \R^n$, $r_0 = \delta_E(x)$, and $k \ge 0$, we let $r_k = 2^k r_0$. For each $1 \le i \le n$, we show that 
\begin{align}
\left| \partial_i  \left( \left|  \nabla D_{K, \mu}(x) \right|^2 \right) \right| & \le \left| \partial_i  \left( \left|  \nabla D_{K, \mu}(x) \right|^2 \right)  -  \partial_i  \left( \left|  \nabla D_{K, \nu}(x) \right|^2 \right)  \right| + \left|\partial_i  \left( \left|  \nabla D_{K, \nu}(x) \right|^2 \right)  \right| \nonumber \\
& \le C \delta_E(x)^{-1} \sum_{l \ge 0} 2^{-(\beta + 1)l} \alpha(y, 2^8 r_0 ) +  \left|\partial_i  \left( \left|  \nabla D_{K, \nu}(x) \right|^2 \right)  \right| \label{thm:usfe_gen_e_ineq}
\end{align}
for $y \in E \cap B(x, 16\delta_E(x))$, and where $\nu = \nu(x)$ is a well-chosen flat measure. Assuming that \eqref{thm:usfe_gen_e_ineq} holds, let us prove the result.

Let $a(y,r)$ denote the function defining the sum on the right-hand side of \eqref{thm:usfe_gen_e_ineq}. That is, $a(y, r)  = \sum_{\ell \ge 0}2^{-(\beta + 1) \ell} \alpha(y, 2^l r)$. From \eqref{thm:usfe_gen_e_ineq}, we have that for $Q \in E$ and $R >0$,
\begin{align}
\int_{B(Q,R)} F_K^2(X) \delta_E(X)^{-n+d} \; dX & \le C \int_{B(Q,R)} \fint_{B(X, 16\delta_E(X)) \cap E} a(y, 2^8 \delta_E(X))^2 \; d\mu(y) \delta_E(X)^{-n+d} \; dX \nonumber \\
& + \sum_{i=1}^n \int_{B(Q,R)} \left|\partial_i  \left( \left|  \nabla D_{K, \nu(X)}(X) \right|^2 \right)  \right|^2  \delta_E(X)^{-n+d+2} \; dX. \nonumber \\
& \equiv I +II. \label{fk_cm}
\end{align}
We first bound $I$ from above by $C R^d$. To do this, decompose $B(Q, R)$ into a disjoint union of Whitney cubes. Switching the order of integration, and summing over the cubes yields 
\begin{multline*}
\int_{B(Q,R)} \fint_{B(X, 16\delta_E(X)) \cap E} a(y, 2^8 \delta_E(X))^2 \; d\mu(y) \delta_E(X)^{-n+d} \; dX \\ \le C \int_{B(Q,R) \cap E} \int_0^R a(y, 2^8 r)^2 \; \dfrac{d\mu(y) dr}{r}.
\end{multline*}
One argues as in Lemma 5.89 in \cite{DFM19} to then show that 
\begin{align*}
\int_{B(Q,R) \cap E} \int_0^R a(y, 2^8 r)^2 \; \dfrac{d\mu(y) dr}{r} & \le C \int_{B(Q,R) \cap E} \int_0^R \alpha(y, 2^8 r)^2 \; \dfrac{d\mu(y) dr}{r}
\end{align*}
and so the Carleson measure estimate (\ref{alpha_cm}) implies $I \le CR^d$.

As for the term $II$ in \eqref{fk_cm}, we require more precise control on $\nu(x)$; we will choose them so that $\nu(x)$ are constant on certain Whitney regions outside $E$ as follows. Let 
\begin{align*}
\Omega_j & = \{ x \in B(Q,R) \; : \; 2^{-j-1}R < \delta_E(x) \le 2^{-j} R \}
\end{align*}
for $j \in \Z, j \geq 0$. Let $\eta \in (0, 1)$ be sufficiently small and fixed (to be determined below), and suppose that $B_i^j = B(x_i^j, \eta \delta_E(x_i^j))$ are a countable collection of balls covering $\Omega_j$ with bounded overlap. That is, we have that $\sum_i \chi_{B_i^j} \le M$ on $\Omega_j$. Such a cover exists by the Besicovitch Covering Theorem and, if we take $\eta$ sufficiently small we may assume that $B_i^j \subset \Omega_{j-1} \cup \Omega_j \cup \Omega_{j+1}$ for each $i, j$. We now need to pick the $\nu(x)$ more carefully.

\medskip

\noindent {\bf Claim:} We claim, and will prove below, that in each $B_i^j$ we can choose a flat measure $\nu_i^j$, supported on $V_i^j$, such that \eqref{thm:usfe_gen_e_ineq} holds for all $x\in B_i^j$ with the measure $\nu_i^j$. Further assume that $\nu_i^j(y) = a_i^j \HD^d|_{V_i^j}(y)$ with $C^{-1} \le a_i^j \le C$, and such that $V_i^j \in A(n,d)$ with $\delta_{V_i^j}(x_i^j) \simeq \delta_{E}(x_i^j)$ for each $i$ and $j$. Of course if $\gamma$ is chosen sufficiently small, then this implies also that $\delta_E \simeq \delta_{V_i^j}$  on $B_i^j$.

\medskip

The $d$-Ahlfors regularity assumption on $E$ implies that $|\Omega_j| \le | \{ x \in B(Q, R) \; : \; \delta_E(x) \le 2^j R \} | \le C R^d (2^{-j} R)^{n-d}$. Thus, there is a constant $C > 0$ (independent of $j$) such that $\#\{B_i^j\} \leq C2^{jd}$. This allows us to estimate $B$ brutally by
\begin{align*}
\int_{B(Q,R)} & \left|\partial_m  \left( \left|  \nabla D_{K, \nu(X)}(X) \right|^2 \right)  \right|^2   \delta_E(X)^{-n+d+2} \; dX \\
& \leq  \sum_{j \geq 0} \sum_{i \in \N} \int_{B_i^j} F_{K, \nu_{i,j}}^2 \delta_{V^i_j}(X)^{-n+d} \; dX \\
&\leq C \sum_{j\geq 0} 2^{jd} \sup_i \int_{B_i^j} F_{K, \nu_{i,j}}^2 \delta_{V^i_j}(X)^{-n+d} \; dX.
\end{align*}
We recall that $F_{K, \nu_i^j}$ and $\delta_{V^i_j}$ are invariant in directions parallel to $V^i_j$, so, arguing as above, there is a Whitney-type cube $\tilde{B}_i^j \subset \mathbb R^{n-d} \cong (V^i_j)^\perp$ (which contains the projection of $B_i^j$ onto $\mathbb R^{n-d}$) such that $$\int_{B_i^j} F_{K, \nu_{i,j}}^2(x) \delta_{V^i_j}(X)^{-n+d} \; dX \leq C_{\eta} (2^{-j}R)^d \int_{\tilde{B}_i^j} F_{K, \nu_{i,j}}^2(y) \delta_{V^i_j}(y)^{-n+d} \; d\mathcal H^{n-d}(y).$$ Letting $\lambda > 1$ be large enough (depending only on $n, d, \eta$ not on $j, i$) we can assume that $\tilde{B}_i^j \subset W_\lambda(2^{-j}R)\cap \mathbb R^{n-d}$. 

Putting everything together (and overestimating the integral on $\tilde{B}_i^j$ by the integral on $W_\lambda(2^{-j}R)$) we get that \begin{equation}\label{e:dealingwiththeflats}\begin{aligned} \int_{B(Q,R)} \left|\partial_m  \left( \left| \nabla D_{K, \nu(X)} \right|^2 \right)  \right|^2   \delta_E^{-n+d+2} \; dX &\lesssim R^d \sum_{j} \int_{(V_i^j)^\perp \cap W_\lambda(2^{-j}R)}F_{K, \nu_{i,j}}^2 \delta_{V^i_j}^{-n+d} \; d\mathcal H^{n-d}\\ &\leq CR^d\sum_j \gamma_{K,\lambda, \alpha}^2(2^{-j}R) \leq C R^d,\end{aligned}\end{equation}
where the penultimate inequality follows from Lemma \ref{lemma:usfeflat} and the final inequality follows from bounding the dyadic sum by the scale invariant integral $\int_0^\infty \gamma_{K, \lambda, \alpha}^2(r)\, dr/r$. Summing over $1 \le m \le n$, we get $II \le C R^d$. In summary, we have bound both terms, $I, II$ in \eqref{fk_cm}, and thus established the USFE, provided that $\nu_i^j$ can be chosen as in the claim above.  

To prove the claim, we can compute explicitly $\left| \partial_i  \left( \left|  \nabla D_{K, \mu}(x) \right|^2 \right)  -  \partial_i  \left( \left|  \nabla D_{K, \nu}(x) \right|^2 \right)  \right|$ as a sum, apply the triangle inequality, and simply estimate each term of the form $|\nabla^j R_{K, \mu}(x) - \nabla^j R_{K, \nu}(x)|$ where $\nabla^j$ is an iterated derivative. Let us consider the simple case, which is when $j=0$.

Let $x \in \Omega$, $r_0= r_0(x) = \delta_E(x)$, and $r_k= r_k(x) = 2^k r_0$ for $k \ge 0$. Let $\phi$ be a fixed smooth bump function so that $0 \le \phi \le 1$, $\phi$ is radial, and $\phi \equiv 1$ on $B(0, 8r_0)$, and $\phi \equiv 0$ outside $B(0, 16 r_0)$. Define $\phi_0 = \phi$, and $\phi_k(x) = \phi(2^{-k}x) - \phi(2^{-k+1}x)$ for $k \ge 1$. Note that $\phi_k$ is supported in $A_k \equiv \overline{B}(0,2^{k+4 }r_0) \setminus B(0, 2^{k+2} r_0)$, and $\sum_{k \ge 0} \phi_k =1$.

Arguing as in \cite{DEMMAGIC} we may choose flat measures $\nu_k = \lambda_k \HD^d|_{P_k}$ where $\lambda_k >0$ and $P_k \in A(n,d)$ so that the following hold:
\begin{align}
\mathcal{D}_{x, 64 r_k}(\mu, \nu_k) \le C\alpha(x, 64r_k)
\end{align}
with constant depending only on $n,d$ and the Ahlfors regularity constant for $\mu$. Moreover, we may choose such measures so that $C^{-1} \le \lambda_k \le C$ and so that $P_k \cap B(x, (3/2) r_k) \ne \emptyset$. Remark also, that if $|x-z| \le \gamma \delta_E(x)$ for $\gamma \in (0, 1/2)$ sufficiently small and fixed, then we have that 
\begin{align*}
\mathcal{D}_{z, 2^{5 + k} \delta_E(z)}(\mu, \nu_k)  \le C \mathcal{D}_{x, 2^{6 + k} \delta_E(x)}(\mu, \nu_k)  \le C \alpha(x, 2^{k+6} \delta_E(x)) \le C \alpha(z, 2^{k+7} \delta_E(z)).
\end{align*}
If $\gamma$ is small enough, we can also guarantee that $P_k \cap B(z, 2\delta_E(z)) \ne \emptyset$ for such $z$. In particular, we can choose the flat measures $\nu_k$ to be constant on $B(x, \gamma \delta_E(x))$ and maintain
\begin{align}
\mathcal{D}_{z, 32 r_k}(\mu, \nu_k) \le C\alpha(z, 128r_k), \; \; P_k \cap B(z, 2 \delta_E(z)) \ne \emptyset
\end{align}
for $z \in B(x, \gamma \delta_E(x))$.

Set $\nu  = \nu(x) = \nu_0$, $P = P(x) = P_0$, and let $z \in B(x, \gamma \delta_E(x))$. We use the $\nu_k$ to estimate $|R_{K, \mu}(z)  - R_{K, \nu}(z)|$.  Without loss of generality, we may assume that $0 \in P \cap B(x, 2r_0)$. A direct computation yields 
\begin{align*}
\left| R_{K, \mu}(z) - R_{K, \nu}(z) \right| & = \left |\sum_{k \ge 0} \int_{A_k} \phi_k(y) \dfrac{K(z-y)}{|z-y|^{d+\alpha}} \; ( d\mu - d\nu) (y) \right|.
\end{align*}
Since $K$ is distance-standard, we have that $\phi_k(y) K(z-y)|z-y|^{d+\alpha}$ is Lipschitz in $y$ with constant at most $C r_{k}^{-d-\beta -1}$ (here, we are using that $\nabla K(w)|w|  \in L^\infty(\R^n)$). Moreover, this function vanishes outside of $B(0, 2^{k+4} r_0) \subset B(0, 2^{k+5} r_0)$. The definition of $\mathcal{D}$ thus gives
\begin{align*}
\left | \int_{A_k} \phi_k(y) \dfrac{K(z-y)}{|z-y|^{d+\alpha}} \; (d\mu - d \nu) (y) \right| \le C r_k^{-\beta} \mathcal{D}_{z, 2^{k+5} r_0}(\mu, \nu).
\end{align*}
By the triangle inequality for $\mathcal{D}$, we have that 
\begin{align*}
\mathcal{D}_{z, 2^{k+5} r_0}(\mu, \nu) \le \mathcal{D}_{z, 2^{k+5} r_0}(\mu, \nu_k) + \sum_{l=1}^k \mathcal{D}_{z, 2^{k+5} r_0}(\nu_l, \nu_{l-1}).
\end{align*}
One argues as in the proof of the equation (5.83) in \cite{DFM19} to obtain $\mathcal{D}_{z, 2^{k+5} r_0}(\nu_l, \nu_{l-1}) \le C \alpha(z, 2^{l+7} r_0)$ since the measures $\nu_i$ are flat, pass near $z$, and approximate $\mu$ well in $B(z, 2^{l+7}r_0)$. It follows that 
\begin{align*}
\left|R_{K, \mu}(z) - R_{K, \nu}(z) \right| & \le C \sum_{k \ge 0 } r_k^{-\beta} \sum_{0 \le l \le k} \alpha(z, 2^{l + 7} r_0) \le C \sum_{l \ge 0} r_l^{-\beta} \alpha(z, 2^{l+7} r_0).
\end{align*}
If $y \in B(z, 16 \delta_E(z))$, then since $\alpha(z, 2^{l+7} r_0) \le C \alpha(y, 2^{l+8} \delta_E(z))$, we have that for all such $y$,
\begin{align*}
\left|R_{K, \mu}(z) - R_{K, \nu}(z) \right| & \le C \sum_{l \ge 0} r_l^{-\beta} \alpha(z, 2^{l+7} r_0) \\ 
& = C r_0^{-\beta} \sum_{l \ge 0} 2^{- \beta l} \alpha(y, 2^{l+8} r_0)   = C \delta_E(x)^{- \beta} \sum_{l \ge 0} 2^{- \beta l} \alpha(y, 2^{l+8} r_0).
\end{align*}
This is the desired estimate for $R_{K, \mu}$, but using the same methods as above, we can show that 
\begin{align*}
\left | \nabla^j R_{K, \mu}(z) - \nabla^j R_{K, \nu}(z) \right| \le C \delta_E(z)^{- \beta - j} \sum_{l \ge 0} 2^{-(\beta + j)l} \alpha(y, 2^{l+8} r_0)
\end{align*}
for each iterated integral with $j = 1,2$, and for $y \in B(z, 16 \delta_E(z))$. Remark as well that since $K$ is distance-standard, we know that $\left| \nabla^j R_{K ,\mu}(z) \right| \le C \delta_E(z)^{-\beta - j}$ for $j = 0,1,2$, just as in the proof of Lemma \ref{lemma:usfeflat}. We can argue as in \cite{DEMMAGIC} to show that $ \left | \nabla^j R_{K ,\nu}(z) \right | \le C \delta_E(z)^{-\beta -j}$ as well, due to the fact that $|y-z| \ge r_0/2$ for $y \in P_0$. 

From here, our estimate follows from the usual process of estimating the terms in 
\begin{align*}
\left| \partial_i  \left( \left|  \nabla D_{K, \mu}(z) \right|^2 \right)  -  \partial_i  \left( \left|  \nabla D_{K, \nu}(z) \right|^2 \right)  \right|
\end{align*}
by brute force i.e. using the bounds $|\nabla^j R_{K, \mu}(z)| \lesssim \delta_E(z)^{-\beta - j}$ and $|\nabla^j R_{K, \nu}(z)| \lesssim \delta_E(z)^{-\beta - j}$, along with the estimates $\left|\nabla^j  R_{K, \mu}(z) -  \nabla^j R_{K, \nu}(z) \right|  \le C \delta_E(z)^{-\beta - j} \sum_{l \ge 0} 2^{-(\beta + j) l} \alpha(y, 2^{l + 8} r_0)$
to show $$\left| \partial_i  \left( \left|  \nabla D_{K, \mu}(z) \right|^2 \right)  -  \partial_i  \left( \left|  \nabla D_{K, \nu}(z) \right|^2 \right)  \right|
\leq C \delta_E(x)^{-1} \sum_{l \ge 0} 2^{-(\beta + 1)l} \alpha(y, 2^8 r_0),$$ with the $\nu$ chosen as in the claim above (this argument works exactly as in the proof of Lemma \ref{lemma:usfeflat}). That concludes our proof of the claim, i.e. that \eqref{thm:usfe_gen_e_ineq} holds with the special choices of $\nu$ described above. The theorem follows. 
\end{proof}

Given Theorem \ref{T1}, it is natural to search for a sufficient condition on $K$ which implies that $D_K$ satisfies \eqref{eqn:good_for_dist} and is easier to verify in practice. In the radial setting this condition is captured by a Dini-type closeness. Recall that this is half of Theorem \ref{thm:rad_usfe} which we restate here for convenience: 

\begin{thm}
Suppose that $K \in C^3(\R^n \setminus \{0\})$ is radial, distance-standard, and $\nabla^3 K(x) |x|^3 \in L^\infty(\R^n)$. Further assume that \begin{align}
\int_0^1 \left(t^{m}  \dfrac{d^m}{dt^m} \left(K(t) - K_0 \right) \right)^2  \; \dfrac{dt}{t} + \int_1^\infty \left(t^{m}  \dfrac{d^m}{dt^m} \left(K(t) - K_\infty \right) \right)^2  \; \dfrac{dt}{t} < \infty \label{cond:rad_usfe}
\end{align}
for some constants $K_0, K_\infty >0$ and for $m=0,1,2$. Then $D_K$ satisfies the USFE in $\Omega = \R^n \setminus \text{spt } \mu$ for any $d$-uniformly rectifiable measure $\mu$. \label{thm:rad_usfe1}
\end{thm}

 Theorem \ref{thm:rad_usfe1} follows from the subsequent lemma and Theorem \ref{T1}:
 
\begin{lemma}\label{lemma:pert_char}
Suppose that $K$ is a distance-standard radial kernel, and $K_0, K_\infty$ are positive constants so that  
\begin{align}
\int_0^1 \left(t^{m}  \dfrac{d^m}{dt^m} \left(K(t) - K_0 \right) \right)^2  \; \dfrac{dt}{t} + \int_1^\infty \left(t^{m}  \dfrac{d^m}{dt^m} \left(K(t) - K_\infty \right) \right)^2  \; \dfrac{dt}{t}  \equiv M_m < \infty \label{rmk:suff_for_*m}
\end{align}
for $m=0,1,2$. Then there is a constant $C_\alpha>0$ depending only on $n, d, \alpha > 0$ and the distance-standard constants of $K$ so that for any $\alpha >0$,
\begin{align}
\int_0^\infty \gamma_{K, 2, \alpha}(r)^2 \; \dfrac{dr}{r} \le C_\alpha( 1 + M_0 + M_1+ M_2).
\end{align}
In particular, $K$ is uniformly good for distances for any exponent.
\end{lemma}

\begin{proof}
Fix $\alpha > 0$. As is often the case we omit in the notation of $R_{K, V, \alpha}$ the dependence on $V$ and $\alpha$ because they are fixed. Furthermore, since $K$ is radial we know that $R_{K,V,\alpha}(x)$ depends only on the $\delta_V(x)$ and, in particular, not on the plane $V \in G(n,d)$. Putting all this together it suffices to estimate: 
\begin{equation}\label{eqn:mod14}
 \sum_{j \in \Z} \inf_{c \in \R} \sum_{m=0}^2 \int_{V^\perp \cap W_4(2^j)}    \left | \delta_{V}(x)^{\alpha + m} \nabla^m \left( R_{K}(x) - c \delta_{V}(x)^{-\alpha} \right) \right|^2 \,  \delta_{V}(x)^{-n+d} \; d\HD^{n-d}(z). 
\end{equation}

We will do the $m= 0$ case, since the other cases follow in the same way. Since $K_0, K_\infty$ are constants, we know that there are constants $c_0, c_\infty >0$ so that $R_{K_0} \equiv c_0 \delta_{V}^{-\alpha}$ and $R_{K_\infty} = c_\infty \delta_{V}^{-\alpha}$. Applying Jensen's inequality and rotational invariance,
\begin{align*}
& \int_{V^\perp \cap B(0,1)} \left | R_{K}(z) - c_0 \delta_{V}(z)^{-\alpha} \right|^2 \delta_{V}(z)^{-n+d+2\alpha }  \; d\HD^{n-d}(z)  \\
& = \int_{V^\perp \cap B(0,1)} \left | R_{K}(z) - R_{K_0}(z) \right|^2 \delta_{V}(z)^{-n+d+2\alpha }  \; d\HD^{n-d}(z)  \\
&  = \int_{V^\perp \cap B(0,1)} \left | \int_{V} \dfrac{K(z-y) - K_0}{|z-y|^{d+\alpha}} \; d\HD^d(y)   \right |^2 \; \delta_{V}(z)^{-n+d+2\alpha }  \; d\HD^{n-d}(z) \\
& \le  \int_{V^\perp \cap B(0,1)} \left ( \int_{V} \dfrac{|K(z-y) - K_0|}{|z-y|^{d+\alpha}} \; d\HD^d(y)   \right )^2 \; \delta_{V}(z)^{-n+d+2\alpha }  \; d\HD^{n-d}(z) \\
& = C \int_{0}^1 \int_{V} \dfrac{(K(z-y) - K_0)^2 }{|z-y|^{d+\alpha}} \; d\HD^d(y) \delta_{V}(z)^{\alpha-1} \;d\delta_V(z) \\
& = C  \int_{0}^1\int_{\rho}^\infty (K(t) - K_0)^2 t^{-d-\alpha+1}\bigl(t^2-\rho^2\bigr)^{\frac{d-2}{2}}\, dt \rho^{\alpha-1}\, d\rho\\
&\leq C\int_0^1 \int_{\rho}^{1}(K(t) - K_0)^2 t^{-d-\alpha+1}\bigl(t^2-\rho^2\bigr)^{\frac{d-2}{2}}\, dt \rho^{\alpha-1}\, d\rho + C\|K - K_0\|^2_{L^\infty}.
\end{align*}
Apply Fubini and a change of variables to get
\begin{align*}
&\int_0^1 \int_{\rho}^{1}(K(t) - K_0)^2 t^{-d-\alpha+1}\bigl(t^2-\rho^2\bigr)^{\frac{d-2}{2}}\, dt \rho^{\alpha-1}\, d\rho  \\
&\leq C\int_0^1 (K(t) - K_0)^2  \int_0^t \Bigl(1-\left(\frac{\rho}{t}\right)^2\Bigr)^{\frac{d-2}{2}}\left(\frac{\rho}{t}\right)^{\alpha} \, \frac{d\rho}{\rho}\, \frac{dt}{t}\\
&=C\int_0^1 (K(t) - K_0)^2 \, \frac{dt}{t},
\end{align*}
since the interior integral in the second line converges for all $d, \alpha > 0$. 

We can estimate  $$\int_{V^\perp\cap B(0,1)^c} \left | R_{K}(z) - c_\infty \delta_{V}(z)^{-\alpha} \right|^2 \delta_{V}(z)^{-n+d+2\alpha }  \; d\HD^{n-d}(z)$$ the same way, and putting all these estimates together we have that $$\begin{aligned} &\sum_{j \in \Z} \inf_{c \in \R} \int_{V^\perp \cap W_4(2^j)}    \left | \delta_{V}(x)^{\alpha }  \left( R_{K}(x) - c \delta_{V}(x)^{-\alpha} \right) \right|^2 \,  \delta_{V}(x)^{-n+d} \; d\HD^{n-d}(z)\\ &\leq C\left(\int_0^1 \left(K(t) - K_0 \right)^2  \; \dfrac{dt}{t} + \int_1^\infty \left(K(t) - K_\infty \right)^2  \; \dfrac{dt}{t} + \|K-K_0\|_{L^\infty}^2 + \|K - K_\infty\|_{L^\infty}^2\right),\end{aligned}$$ for any choice of $K_0, K_\infty$. Notice that each of the integrals above converge by assumption. Moreover, for the integrals to converge it must be the case that $K_0, K_\infty \leq 4 \|K\|_{L^\infty}$ and so we can bound $\|K- K_0\|_{L^\infty}, \|K-K_\infty\|_{L^\infty}$ by the distance-standard constants of $K$. Thus we have completed our proof when $m = 0$ and, as mentioned above, the rest of the argument follows similarly. 
\end{proof}

As was the case for non-tangential limits, our condition for general kernels $K$ is less clean, due to the richness of the family of distance-exact kernels. The following Lemma is proven is much the same way as above, so we omit the argument. 
\begin{lemma}\label{lemma:general_suff}
Suppose that $K$ is distance-standard, and $K_0, K_\infty$ are $(d,\alpha)$-distance-exact kernels. Then there is a constant $C >0$ depending only on $n, d, \alpha$ and the distance-standard constants for $K, K_0$ and $K_\infty$ so that 
\begin{align}
\int_0^\infty \gamma_{K,2, \alpha}(r)^2 \; \dfrac{dr}{r} \le C \left ( 1 + \int_0^1 \theta_{K, K_0}( r)^2 \; \dfrac{dr}{r} + \int_1^\infty \theta_{K, K_\infty}(r)^2  \; \dfrac{dr}{r}) \right)
\end{align}
where for $r \in (0,1)$, we define
\begin{align*}
\theta_{K, K_0}(r)^2 \equiv  \sup_{V \in G(n,d)}  \sum_{m=0}^2 \int_{(V^\perp \cap W_2(r)) \times (B(0,1) \cap V)}  \left| |x|^m \nabla^m (K(x)  \right . & \left . - (K_0)(x)) \right|^2  \\
& \times |x|^{-d-\alpha} \delta_{V}(x)^{-n+d+\alpha} \; dx 
\end{align*}
and for $r > 1$ we define
\begin{align*}
\theta_{K, K_\infty}(r)^2 \equiv  \sup_{V \in G(n,d)}  \sum_{m=0}^2 \int_{(V^\perp \cap W_2(r)) \times V}  \left| |x|^m \nabla^m (K(x) \right . & -  \left . (K_\infty)(x)) \right|^2  \\
& \times  |x|^{-d-\alpha} \delta_{V}(x)^{-n+d+\alpha} \; dx. 
\end{align*}
In particular, if $\int_0^1 \theta_{K, K_0}(r)^2 dr/r  + \int_1^\infty \theta_{K, K_\infty}(r)^2  dr/r < \infty$, then $K$ is uniformly good for distances for any exponent $\alpha >0$.
\end{lemma}

As a final goal of this subsection, we now investigate to what extent the sufficient conditions of Theorems \ref{T1}, \ref{thm:rad_usfe1} and Lemma \ref{lemma:general_suff} are sharp. Due to our inability to make the results of Theorem \ref{thm:dist_exact_radial} quantitative; that is to say, our lack of a theorem which says ``If $D_K$ is quantitatively close to being the distance than $K$ is quantitatively close to being a distance standard kernel", there is little hope for us to find necessary quantitative conditions on the kernels $K$, even in the radial setting. However, we are able to find some necessary conditions on $D_K$ for the USFE to hold outside of planes (and thus uniformly rectifiable sets).

We begin with a simple computation which shows that a condition not so far from being uniformly good for distances (e.g. Definition \ref{defn:gamma}) is necessary for satisfying the USFE outside of all planes: 

\begin{lemma}\label{l:poincare}
If $K$ is a distance standard kernel such that $F_K$ satisfies the USFE outside of all $d$-planes then for each $E\in G(n,d)$ there are constants $c_k$ such that
\begin{align*}
\sum_{k    \in \Z} \int_{A_k}|  |\nabla D_K|(z)  -  c_k|\nabla  \delta_E(z)| |^2 \delta_E(z)^{- n + d} \; d\HD^{n-d}(z) < \infty,
\end{align*}
where $A_k = E^\perp \cap \left( B(0, 2^k) \setminus B(0, 2^{k-1})\right)$.
\end{lemma}

Above we have suggestively written $|\nabla  \delta_E(z)|$ instead of 1 with an eye towards future applications in the radial setting. Before we prove the lemma let's make some remarks: 

\begin{rmk}
The condition in Lemma \ref{l:poincare} differs from Definition \ref{defn:gamma} in two main ways. First, it gives no control on the second or 0th order derivatives of $D_K$. While the USFE itself is control on $\nabla |\nabla D_K|$ (i.e. part of the second derivative) and similar arguments to the ones below give a notion of control on the 0th derivative, we were not able to show that control on all the second derivatives of $D_K$ is necessary for the USFE outside of planes. 

Secondly, Definition \ref{defn:gamma} requires control which is uniform over planes whereas Lemma \ref{l:poincare} shows only non-uniform control is necessary. To get uniform control for each kernel $K$ one would need to build a set $E$ which is flat in the ``worst possible" direction for a Carleson prevalent set of Whitney-type regions.
Whether such a set exists for every kernel is an interesting question, but outside the scope of this article. 
\end{rmk}

\begin{proof}[Proof of Lemma \ref{l:poincare}]
  Let $E \in G(n,d)$. Since $F_K$ satisfies the USFE on $\Omega$, and since $F_K$ and $\delta_E$ are translation invariant with respect to $E$, then necessarily
      \begin{align*}
          \int_{E^\perp} F_K(z)^2 \delta_E(z)^{-n+d} \; d\HD^{n-d}(z) < \infty.
      \end{align*}
      With $A_k = B(0, 2^{k+1}) \setminus B(0, 2^k) \cap E^\perp$, we have that 
      \begin{align*}
          \sum_{k \in \Z} \int_{A_k} F_K(z)^2 \delta_E(z)^{-n + d} \; d\HD^{n-d}(z) & = \sum_{k \in \Z} \int_{A_k} |\nabla | \nabla D_K(z)|^2|^2  \delta_E(z)^{-n + d +2 } \; d\HD^{n-d}(z) < \infty. 
      \end{align*}
      With $r_k = 2^{-k}$ and $c_k = \int_{A_k} |\nabla D_K(z)|^2 \; d\HD^{n-d}(z) / |A_k|$,
      \begin{multline*}
          \int_{A_k} | |\nabla D_K(z)|^2 - c_k|^2 \delta_E(z)^{-n + d} \; d\HD^{n-d}(z)  = \int_{A_1} | | \nabla D_K( r_k z)|^2 - c_k|^2 \delta_E(z)^{-n +d} \; d\HD^{n-d}(z) \\
           \le C(n,d) \int_{A_1} | | \nabla D_K( r_k z)|^2 - c_k|^2  \; d\HD^{n-d}(z)   \le C(n,d) \int_{A_1} |\nabla_z |\nabla D_K(r_k z)|^2 |^2 \; d\HD^{n-d}(z)  \\
           = C(n,d) r_k^2 \int_{A_1} |\nabla | \nabla D_K(r_k z)|^2 |^2 \; d\HD^{n-d}(z)  \\ \le C(n,d) \int_{A_k} |\nabla | \nabla D_K(z)|^2|^2 \delta_E(z)^{-n +d + 2} \; d\HD^{n-d}(z),
      \end{multline*}
      where we have employed the Poincar\'e inequality in the second line. This shows that 
      \begin{align*}
          \sum_{k \in \Z} \int_{A_k} | | \nabla D_K(z)|^2 - c_k|^2 \delta_E(z)^{-n + d} \;d\HD^{n-d}(z) < \infty,
      \end{align*}
as desired. 
\end{proof}

Since radial kernels do not depend on the direction (nor the plane itself), we can strengthen Lemma \ref{l:poincare} in the radial setting:

\begin{cor}\label{c:almostconverse}
If $K$ is a radial, distance-standard kernel so that $D_{K, \HD^d|_E}$ satisfies the USFE for each $E \in G(n,d)$, then necessarily we have that 
\begin{align*}
\sum_{k    \in \Z} \int_{A_k}|  \nabla D_K(z)  -  \nabla  c_k\delta_E(z) |^2 \delta_E(z)^{- n + d} \; d\HD^{n-d}(z) < \infty,
\end{align*}
where $A_k = E^\perp \cap \left( B(0, 2^k) \setminus B(0, 2^{k-1})\right)$.
\end{cor}
\begin{proof}
Lemma \ref{l:poincare} shows that for each $E$ if $F_K$ satisfies the USFE outside of $E$ we have constants $c_k$ such that \begin{align*}
\sum_{k    \in \Z} \int_{A_k}| | \nabla D_K(z)|^2  - c_k|^2 \delta_E(z)^{- n + d} \; d\HD^{n-d}(z) < \infty
\end{align*}
Since $D_K$ is radial, we know that $\nabla D_K(x)$ points the same direction as $\nabla \delta_E(x)$, and that the $c_k$ do not depend on $E$, giving the desired result. 
\end{proof}

\begin{thm}[A partial converse] \label{thm:necforusfe}
 Suppose $E \in G(n,d)$ and $K$ is a distance-standard kernel such that $F_{K, \HD^d|_E}$ satisfies the USFE. Then necessarily there are constants $a_{(*)} > 0$ so that 
\begin{align*}
   \lim\limits_{|x| \ra (*), \;  x \not \in E} |D_K(x) - a_{(*)}\delta(x)| \delta(x)^{-1} \ra 0.
\end{align*} 
where $(*)$ stands for $0$ or $\infty$. 

Furthermore if $\nabla^m K(x) |x|^m \in L^\infty(\R^n)$ for $0 \le m \le k$ we can ensure, for $0 \le m \le k-1$, that
\begin{align*} 
   \lim\limits_{|x| \ra (*), \;  x \not \in E} |\nabla^m D_K(x) - a_{(*)}\nabla^m \delta_E(x)| \delta_E(x)^{1-m} \ra 0.
\end{align*} 

Finally, if $K$ is radial, then we also have that $K_\lambda \ra c_{(*)}$ in $C^{k-1}_{loc}(\R^n \setminus \{0\})$ as $\lambda  \ra (*)$ for $(*) = 0, \infty$ and some positive constants $c_0, c_\infty$.
\end{thm}
\begin{proof}
           
      Define the functions $f_N(x) : E^{\perp} \setminus \{ 0 \} \ra \R$ by $f_N(z) = D_K(2^{-N} z) 2^{N}$. Since $K$ is a distance-standard kernel, we have that 
      \begin{enumerate}[(a)]
          \item $|\nabla f_N(z)| \le C$ for all $z \in (B(0,1) \cap E^{\perp}) \setminus \{0\}$,
          \item $f_N(0) = 0$ for each $N \in \N$,
          \item $|f_N(z)|$ is uniformly bounded on $B(0,1) \cap E^{\perp}$. 
      \end{enumerate}
      In particular, by Arzela-Ascoli, we may choose a subsequence (still labeled $f_N$) such that $f_N \ra f$ uniformly on $B(0,1) \cap E^{\perp}$. Now let $E_k = B(0, 2^k) \setminus B(0, 2^{-k}) \cap E^{\perp}$. Note that for each fixed $k$, we may choose a constant $C$ so that for each $N$ and for any $x \in E_k$,
      \begin{enumerate}[(a)]
          \item $|f_N(z)| \le C \delta_E(z) \le C 2^k$,
          \item $|\nabla f_N(z)| \le C$, 
          \item $|\nabla^2 f_N(z)| \le C \delta_E(z)^{-1} \le C 2^k$.
      \end{enumerate}
This allows us to extract a subsequence of $f_N$ for which $f_N \ra g_k$ in $C^1(E_k)$. By a diagonalization argument, we may then as well assume that $f_N \ra g$ in $C^1_{loc}(E^\perp \setminus \{0\})$ for some function $g \in C^1(E^{\perp} \setminus \{0\})$. 
      
      We now show that $|\nabla g|$ is constant. For each fixed $k$, note that
      \begin{align*}
          \int_{A_k} ||\nabla f_N(z)|^2 - c_{k-N}|^2  \; d\HD^{n-d}(z) & = \int_{A_k} | |\nabla D_K(2^{-N} z)|^2 - c_{k-N}|^2 \; d\HD^{n-d}(z) \\
          & = 2^{N(n-d)} \int_{A_{k - N}} | | \nabla D_K(z)|^2 - c_{k-N}|^2 \; d\HD^{n-d}(z) \\
          & \le C(k, n ,d) \int_{A_{k-N}} | |\nabla D_K(z)|^2 - c_{k-N}|^2 \delta_E(z)^{-n + d} \; d\HD^{n-d}(z).
      \end{align*}
      Since 
      \begin{align*}
          \sum_{k    \in \Z} \int_{A_k}| | \nabla D_K(z)|^2  - c_k|^2 \delta_E(z)^{- n + d} \; d\HD^{n-d}(z) < \infty,
      \end{align*}
      sending $N \ra \infty$, we see that 
      \begin{align*}
          \lim_{N \ra \infty }\int_{A_k} ||\nabla f_N(z)|^2 - c_{k-N}|^2  \; dz  = 0. 
      \end{align*}
      Finally, using that $f_N \ra g$ in $C^1_{loc}(E^{\perp} \setminus \{0\})$, we see that this implies both that $\tilde{c} = \lim_{k \ra -\infty} c_k$ exists, and that $|\nabla g(z)| = \tilde{c}$ on $A_k$. Since $k$ is arbitrary we see that $|\nabla g(z)| = \tilde{c}$ on $E^\perp \setminus \{0\}$. 
      
      Recall now that $ (1/C) \delta_E(z) \le f_N(z) \le C \delta_E(z)$, so that in particular, $(1/C) \delta_E(z) \le g(z) \le C \delta_E(z)$. It follows that, $g(0) =0$, $g(z) > 0$ for each $z \in E^{\perp} \setminus \{0\}$, and $\tilde{c} \ne 0$. Since $g \in C^1(E^\perp \setminus \{0\})$, Section 3 in \cite{DEMMAGIC} implies that necessarily, $g(z) = a_0 \delta_E(z)$ for some constant $a_0 > 0$. 
      
      Now since $f_N \ra g$ uniformly on $B(0,1) \cap E^\perp$, we see that this implies that for each $\epsilon >0$, if $N$ is large enough, then for any $|z| \le 1$, $z \in E^\perp$,
      \begin{align*}
          |f_N(z) - a_0 \delta_E(z)| < \epsilon,
      \end{align*}
      i.e.,
      \begin{align*}
          |D_K(2^{-N} z) 2^N - a_0 \delta_E(z)| & = 2^N |D_K(y) - a_0 \delta_E(y)| < \epsilon
      \end{align*}
      where $y = 2^{-N} z$. It follows that $\lim_{|z| \ra 0}  \delta_E(z)^{-1}|D_K(z) - a_0 \delta_E(z)| = 0$. This proves the claim for $(*) = 0$, but the proof for $(*) = \infty$ is essentially the same.

To prove convergence for the higher derivatives we again restrict ourselves to $(*) = 0$, since the argument for $(*) = \infty$ is almost identical. Let $\lambda_i>0$ be any sequence converging to zero. For $\lambda >0$, define $K_\lambda(x) = K(\lambda x)$. Now define the sequence of functions $g_i(x)  = D_{K_{\lambda_i}}(x)$ for $x \not \in E$. By Arzela-Ascoli, we may assume that up to a subsequence, $K_{\lambda_i} \ra K_\infty$ in $C^{k-1}_{loc}(\R^n \setminus \{0\})$ for some function $K_\infty \in C^{k-1}(\R^n \setminus \{0\})$ which is distance standard. By Lemma \ref{kernellemma}, it follows that $D_{K_i} \ra D_{K_\infty}$ in $C^{k-1}_{loc}(\R^n \setminus E)$. In particular, we have that $D_{K_{\lambda_i}} \ra D_{K_\infty}$ uniformly in $C^{k-1}(E^\perp \cap \overline{ (B(0,2) \setminus B(0,1))})$.

As in the computations following Lemma \ref{kernellemma}, we have that 
\begin{align*}
D_{K_{\lambda_i}}(x) & = D_K(\lambda_i x) / \lambda_i
\end{align*}
and thus
\begin{align*}
\nabla^m D_{K_{\lambda_i}}(x) & = \nabla^m D_K(\lambda_i x) \lambda_i^{m-1}
\end{align*}
for $ 0 \le m \le k-1$.  When $m = 0$, we know by above that the right-hand side necessarily converges to $a_0 \delta_E(x)$ as $i \ra \infty$ uniformly for $x \in E^\perp \cap \overline{ (B(0,2) \setminus B(0,1))}$. It follows that $D_{K_\infty}(x) \equiv a_0 \delta_E(x)$ on $E^\perp \cap \overline{ (B(0,2) \setminus B(0,1))}$. Thus we have that
\begin{align*}
\nabla^m D_K(\lambda_i x) \lambda_i^{m-1} \ra \nabla^m a_0 \delta(x)
\end{align*}
uniformly in $E^\perp \cap \overline{ (B(0,2) \setminus B(0,1))}$ as $i \ra \infty$. Since $\lambda_i$ was arbitrary, then necessarily we have 
\begin{align*}
\lim_{\lambda \da 0} \sup_{x \in E^\perp \cap \overline{ (B(0,2) \setminus B(0,1))}} | \nabla^m D_K(\lambda x) \lambda^{m-1} - \nabla^m a_0 \delta_E(x)| = 0.
\end{align*}
Unraveling the definition of the limit above gives the desired result, since $D_K$ and $\delta_E$ are translation invariant with respect to $E$. 

To prove the last claim, remark that we have showed that for any sequence $\lambda_i \ra (*)$, there is a subsequence $\lambda_{i_j}$ so that $K_{\lambda_{i_j}} \ra K_\infty$ in $C^{k-1}_{loc}(\R^n \setminus \{0\})$ and $K_\infty$ is such that $D_{K_\infty} \equiv a_{(*)} \delta_E$. Since $K$ is radial, then so is $K_\infty$. But then Theorem \ref{thm:dist_exact_radial} implies that $K_\infty$ is a positive constant (and the constant is determined by $a_{(*)}$). Hence the claim is proved.
\end{proof}

\subsection{USFE imply uniform rectifiability}
\label{sec:usfeimpliesur}

The goal of this subsection is to understand for which distance-standard kernels $K$, the USFE for $D_{K, \mu}$ implies that $\mu$ is uniformly rectifiable. As with non-tangential limits this will be true for essentially all distance-standard kernels. We follow the techniques from \cite{DEMMAGIC}, but first let us introduce some notation.

\begin{defn}
If for each $\epsilon >0$, the set 
\begin{align*}
    Z(\epsilon) = \{ x \in \Omega \; ; \; F_K(x) > \epsilon \}
\end{align*} is a Carleson set, then we say that $F_K$ satisfies the weak USFE. 
\end{defn}

By Chebyshev's inequality, it is not hard to see that the USFE imply weak USFE. Instead of proving uniform rectifiability directly, we will show 
that the weak USFE imply that $E$ satisfies the Bilateral Weak Geometric Lemma (BWGL). Along with Ahlfors regularity, this condition characterizes uniform rectifiability \cite{DSUR}. 

In what follows, there is one additional assumption needed on $K$ to ensure the blow-ups of $F_K$ are well-behaved. Namely, in addition to assuming that $K$ is distance-standard, we assume throughout the entirety of this section that $\nabla^3 K(x)|x|^3$ is bounded in order to have control on the blowups of $F_K$.

\begin{lemma}
Let $n \ge 2$, let $d < n$, and assume $E \subset \R^n$ supports a $d$-Ahlfors regular measure $\mu$. Also, assume that $K$ is a distance-standard kernel with the additional assumption that $K \in C^{3} (\R^n \setminus \{0\})$ with $\|\nabla^3K(x) |x|^3\|_\infty < \infty$. If $M \ge 1$ and $x \in \Omega$, define 
\begin{align*}
    W(x) = W_M(x) = \{ y \in \Omega \cap B(x, M\delta(x)) \; ; \; \delta(y) \ge M^{-1}\delta(x) \}.
\end{align*}
Define the bad sets $\mathcal{B}(\eta) = \mathcal{B}_M(\eta)$ by
\begin{align*}
    \mathcal{B}_M(\eta) & = \{ x \in \Omega \; ; \; F_K(y) 
    \ge \eta \text{ for some } y \in W_M(x)\}.
\end{align*}
If $Z(\epsilon)$ is a Carleson set, then for each $M$ sufficiently large, $\mathcal{B}_M(3 \epsilon)$ is a Carleson set as well. \label{whitneylemma}
\end{lemma}
\begin{proof}
In \cite{DEMMAGIC} this lemma is proven for $F_1$. But the only property of $F$ used there is the following continuity assumption: for any $\epsilon >0$, there is a $\tau > 0$ small enough so that 
\begin{align}
    |F(x)- F(x')| \le \epsilon \text{ whenever } x,x' \in \Omega \text{ with } |x - x'| \le 4 \tau \delta(x). \label{scaleuc}
\end{align}
It thus suffices to check this condition, which we do by contradiction.

Suppose that there is an $\epsilon_0>0$ and a sequence $\tau_i \da 0$ so that for each $i$, one can find points $x_i, y_i \in \Omega$ so that $|x_i - y_i| \le 4 \tau_i \min \{\delta(x_i), \delta(y_i) \}$, but $|F_{K, \mu}(x_i) - F_{K, \mu}(y_i)| \ge \epsilon_0$. Let $Q_i \in E$ be such that $|x_i - Q_i| = \delta(x_i) \equiv R_i$, let $X_i = (x_i - Q_i)/R_i$, and let $Y_i = (y_i - Q_i)/ R_i$. Define $E_i = (E -Q_i)/ R_i$, and $\mu_i$ as in the remarks following Lemma \ref{kernellemma}. Up to a subsequence, we may as well assume that $\mu_i \rightharpoonup \mu_\infty$ and $E_i \ra E_\infty$ locally in the Hausdorff metric. Since $|X_i| = 1$ for each $i$, we may also assume $X_i \ra X_\infty$ for some $|X_\infty| = 1$.

Next, we set $K_i(x) = K(R_i x)$. Remark that since $K$ is distance-standard, we know that $K_i, \nabla K_i$ are equicontinuous and uniformly bounded on compact subsets of $\R^n \setminus \{0\}$. Since $\nabla^3 K$ exists and $\| \nabla^3K(x)|x|^3\|_\infty < \infty$, it follows that $\nabla^2 K_i$ are also equicontinuous and uniformly bounded on compact subsets of $\R^n\setminus \{0\}$. By Arzela-Ascoli, we may assume up to a subsequence that $K_i \ra K_\infty$ in $C^2_{loc}(\R^n \setminus \{0\})$ for some distance-standard kernel $K_\infty \in C^{2} (\R^n \setminus \{0\})$. Using \eqref{dkscale} we have  
\begin{align*}
    F_{K_i, \mu_i}(X_i)  = F_{K, \mu} (x_i) \quad\mbox{and} \quad 
    F_{K_i, \mu_i}(Y_i)  = F_{K, \mu} (y_i).
\end{align*}
In view of Lemma \ref{kernellemma}, we know that $D_{K_i, \mu_i} \ra D_{K_\infty, \mu_\infty}$ in $C^2_{loc}(\R^n \setminus \{0\})$, and thus $F_{K_i, \mu_i} \ra F_{K_\infty, \mu_\infty}$ uniformly on compact subsets of $\R^n \setminus \{0\}$.  Recall now that 
\begin{align*}
    \epsilon_0 & \le |F_{K, \mu}(x_i) - F_{K, \mu}(y_i)|   = |F_{K_i, \mu_i} (X_i) - F_{K_i, \mu_i}(Y_i) |.
\end{align*}
Since $|X_i - Y_i| = R_i^{-1}|x_i - y_i| \le 4 \tau_i$, and $X_i \ra X_\infty$, we have that $Y_i \ra X_\infty$ as well. Since $F_{K_i, \mu_i} \ra F_{K_\infty, \mu_\infty}$ uniformly on compact subsets of $\R^n \setminus \{0\}$, we see that 
\begin{align*}
    \limsup_{i \ra \infty}|F_{K_i, \mu_i}(X_i) - F_{K_i, \mu_i}(Y_i)| = 0,
\end{align*}
a contradiction. It follows that the function $F_K$ satisfies (\ref{scaleuc}), and so one may argue exactly as in \cite{DEMMAGIC} to prove the lemma.
\end{proof}

We may now follow the blowup argument to obtain the main lemma in this section.
\begin{lemma}
    Assume $K$ is as in Lemma \ref{whitneylemma}. For each choice of $0 < d < n$, $\alpha >0$, an Ahlfors regularity constant $C_0$, and constants $\eta >0$ and $N \ge 1$, we can find $M \ge 1$ and $\epsilon >0$ such that if $\mu$ is Ahlfors regular (of dimension $d$, constant $C_0$, and support $E \subset \R^n$), and if $x \in \Omega \setminus \mathcal{B}_M(3 \epsilon)$, then $d$ is an integer and there is a $d$-plane $P$ such that $d_{x, N\delta(x)}(E, P) \le \eta$. 
\end{lemma}
\begin{proof}
Again, following \cite{DEMMAGIC}, we argue by contradiction. That is, assume that there is such a choice of $0 < d < n$, $\alpha >0$, Ahlfors regular constant $C_0$ and parameters $\eta >0$ and $N \ge 1$ so that for each $M_i = 2^i$ and $\epsilon_i = 2^{-i}$, there is a $d$-Ahlfors regular measure $\mu_i$ (with constant at most $C_0$) with support $E_i$ and point $x_i \in \Omega_i \setminus \mathcal{B}_{M_i}(3\epsilon_i)$ that violate the hypotheses of the lemma. We now proceed with a blow-up argument as in Lemma \ref{whitneylemma}. 

Let $Q_i$ be a point that attains $\delta_{E_i}(x_i)$, and define $R_i \equiv \delta_{E_i}(x_i)$ with $X_i  \equiv (x_i - Q_i)/R_i$. Arguing as in the previous lemma, define $\tilde{E_i} \equiv (E_i - Q_i)/R_i$, with supporting $d$-Ahlfors measure $\tilde{\mu}_i(A) \equiv \mu_i(R_i A + Q)/R_i^d$ which has Ahlfors regularity constant bounded by some constant time $C_0$. Without loss of generality, we may assume that $\tilde{\mu_i} \rightharpoonup \mu_\infty$, and $\tilde{E}_i \ra E_\infty$ locally in the Hausdorff metric. For the rescaled kernels $K_i(x) \equiv K(R_i x)$, we have up to a subsequence that $K_i \ra K_\infty$ in $C^2_{loc}(\R^n \setminus \{0\})$. By Lemma \ref{kernellemma}, we thus have that $F_{K_i, \tilde{\mu}_i} \ra F_{K_\infty, \mu_\infty}$ uniformly on compact subsets of $\Omega_\infty$. 
 
 Recall that $x_i \not \in \mathcal{B}_{M_i}(3 \epsilon_i)$, and thus for each $i$,
 \begin{align*}
     F_{K, \mu_i}(y) \le  3(2^{-i}) \text{ for all } y \in \Omega_i \cap B(x_i, 2^i \delta_{E_i}(x_i)) \text{ with } \delta_{E_i}(y) \ge 2^{-i} \delta_{E_i}(x_i).
 \end{align*}
 Computing as in the previous lemma, we have for $ y \in \Omega_\infty$ and $i$ large enough so that $y \in \tilde{\Omega}_i,$
 \begin{align*}
     F_{K_i, \tilde{\mu}_i} ( y) = F_{K, \mu_i}(R_i y + Q_i).
 \end{align*}
 Recalling that $R_i = \delta_{E_i}(x_i)$, we see that for $i$ sufficiently large, $R_i y + Q_i \in B(x_i, 2^i \delta_{E_i}(x_i))$. Moreover, since $\delta_{E_i}(R_i y + Q) = R_i \delta_{\tilde{E}_i}(y)$, we have again for $i$ sufficiently large that $\delta_{E_i}(R_i y + Q_i) \ge 2^{-i} \delta_{E_i}(x_i)$. In particular, for all $i$ sufficiently large, one has that 
 \begin{align*}
     F_{K_i, \tilde{\mu}_i} (y) & = F_{K, \mu_i}(R_i y + Q_i)  \le 3 (2^{-i}). 
 \end{align*}
 Letting $i \ra \infty$, we see that $F_{K_\infty, \mu_\infty} (y)= 0$ for each $y \in \Omega_\infty$. In particular, we have that $|\nabla D_{K_\infty, \mu_\infty}|$ is constant on each connected component of $\Omega_\infty$. Since $D_{K_\infty, \mu_\infty}$ is comparable to $\delta_{E_\infty}$, this constant must be nonzero. \cite[Corollary 3.2]{DEMMAGIC} implies that $d$ is an integer and $E_\infty$ is a $d$-plane.
 
 We have thus obtained that $E_i$ converges, in the Hausdorff distance sense, to $E_\infty$ which is a $d$-plane. Hence, for $i$ sufficiently large, we know that $d_{0, N}(E_i, E_\infty) \le \eta$, which contradicts our starting assumption. The lemma is thus proven. 
\end{proof}

With these lemmata in hand, proof that the (weak) USFE implies uniform rectifiability now proceeds precisely as in the proof of \cite[Theorem 4.1]{DEMMAGIC}. We restate the theorem: 

\begin{thm}
Let $K$ be a $C^3(\R^n \setminus \{0\})$ distance-standard kernel such that $\|\nabla^3 K(x)|x|^3\|_{\infty} < \infty $. Let $n \ge 1$ be an integer, and let $0 < d < n$ be given. Let $\mu$ be a $d$-Ahlfors regular measure with support $E$, and let $\alpha >0$ be given. If for each $\epsilon >0$, the set $Z(\epsilon)$ is a Carleson set, then $d$ is an integer and $E$ is uniformly rectifiable. \label{thm:weakusfe_ur}
\end{thm}

\appendix

\section{Proofs of Theorems \ref{thm:dist_exact_ur_usfe} and \ref{thm:dist_exact_ntl_rect}}\label{appendix}

\subsection{Proof of Theorem \ref{thm:dist_exact_ur_usfe}}\label{proof:dist_exact_ur_usfe}
\begin{proof}
One direction of the statement is proved by Theorem \ref{thm:ntl_rect}, since the theorem is stated and proved for general distance-standard kernels. 

For the other direction, we proceed exactly as in the proof of Theorem \ref{thm:rect_ntl}. Using the notation as in this proof, the key difference is that up to a subsequence, we have that $K_{R_i} = K(R_i \, \cdot)$ converges in $C^1_{loc}(\R^n \setminus \{0\})$ to a kernel $K_\infty$ that is in $C^1(\R^n \setminus \{0\}) \cap L^\infty(\R^n)$ so that $\nabla K(x) |x| \in L^\infty(\R^n)$. Since $K_{R_i}$ is $(d, \alpha)$-distance-exact for each $R_i$, using Lemma \ref{kernellemma}, we see that $K_\infty$ is distance-exact. In particular, the crux of the argument, that $|\nabla D_{K_\infty, \mu_\infty}|$ is a positive constant outside $\Omega_\infty$, is still true, and the remainder of the proof holds true with the kernel $c_\infty$ replaced with $K_\infty$. 
\end{proof}

\subsection{Proof of Theorem \ref{thm:dist_exact_ntl_rect}} \label{proof:dist_exact_ntl_rect}
\begin{proof}
As above, one direction is proved by Theorem \ref{thm:weakusfe_ur} and the preceding lemmas, since the statements and proofs are applicable to general distance-standard kernels $K$ satisfying $\nabla^3K(x) |x|^3 \in L^\infty(\R^n)$. 

As for the other direction, this follows from Theorem \ref{T1} since $(d,\alpha)$-distance exact kernels (obviously) satisfy the uniformly good for distances condition, since $\gamma_{K, \lambda, \alpha} \equiv 0$. 
\end{proof}

\section{Computations for radially-invariant distance exact kernels}\label{appendix:homog_ker}

Here we provide justification to the claim mentioned in Section \ref{sec:class} that depending on the choice of parameters $n,d$ and $\alpha$, examples of continuous 0-homogeneous $(d,\alpha)$-distance exact kernels may exist or may be shown to not exist. Throughout the rest of this section, we will assume $K \in C(\R^n \setminus \{0\})$ is 0-homogeneous and abuse notation and use $K$ to refer to both the kernel on all of $\mathbb R^n \setminus \{0\}$ and to the kernel's restriction to $\mathbb S^{n-1}$ (which, by homogeneity, completely determines the function).

For a $d$-plane $E \subset \R^n$, let $P_E: \R^n \ra E$ be the orthogonal projection onto $E$. In addition, for $x \not \in E$, we define the half $d$-arc induced by $E$ and $x$ to be the subset $H^d(E, x) \subset \sphere^{n-1}$ defined by 
\begin{align*}
      H^d(E, x) & = \{ y/|y| \; : \; y \in E - x \}. 
\end{align*}
Geometrically, $H^d(E, x)$ is where the vectors $y-x$ for $y \in E$ intersect $\sphere^{n-1}$.

\begin{lemma}[A useful change of variables] Let $n,d \in \N$ with $1 \le d < n$ and let $\alpha >0$. Let $E \in G(n,d)$ and suppose $K \in L^\infty(\sphere^{n-1})$. If $\delta_E(x_0) = 1$, then
\begin{align*}
    \int_{E} \dfrac{K((x_0-y)/|x_0-y|)}{|x_0-y|^{d + \alpha}} \; d \HD^d  (y)& = \int_{H^d(E, x_0)} K(-w) (w \cdot w_0)^{\alpha -1 }\; d \HD^d(w)
\end{align*}
where $H^d(E, x_0) \subset \sphere^{n-1}$ is the half $d$-arc induced by $E$ and $x_0$ as above, and 
$$w_0 = w_0(x_0) = (P_E(x_0) -x_0)/|P_E (x_0) - x_0|.$$   \label{cov}
\end{lemma}

\begin{proof}
      By translation invariance, assume that $x_0=0$. After pre-composing with a harmless rotation we may also assume that $E$ is given by the plane $\{(y_1, y_2, \dotsc, y_d, 0, 0, \dotsc, 0, -1)\} \subset \R^n$. Consider the map $\Psi$ from the open $d$-ball in $\R^n$,
      \begin{align*}
          U^d & = \{ x = (x_1, x_2, \dotsc, x_d, 0) \; : \; x \in \R^n, |x| < 1 \}
      \end{align*}
      to the plane $E$ given by
      \begin{align*}
          (x_1, x_2, \dotsc, x_d, 0) & \ra \dfrac{\tilde{x}}{\sqrt{1 - |x|^2}}, \\
          \tilde{x} & = (x_1, x_2, \dotsc, x_d, 0, \dotsc, 0, - \sqrt{1 - |x|^2}).
      \end{align*}
Notice that $\tilde{x} \in H^d(E, x)$ since $\tilde{x} \in \sphere^{n-1}$ and $\lambda \tilde{x} + x \in E$ for $\lambda = \left( \sqrt{1-|x|^2} \right)^{-1}$. Moreover, it is easy to check that this map $\Psi$ is a bijective mapping from $U^d$ to $E$. Viewing $E$ as a parametrized $d$-manifold, we can compute $R_K(0)$ as an integral over $U^d$. To do this though, we need to compute
\begin{align*}
V(D\Psi(x)) & \equiv \left |\det \left( D\Psi(x)^T D\Psi(x) \right) \right|^{1/2}.
\end{align*}
One can easily check that $V(D \Psi(x))$ is identically equal to $|\det A(x)|$ for the $d \times d$ matrix $A = (a_{ij}(x))$ given by
\begin{align*}
a_{ij}(x) & = \left( \dfrac{\delta_{ij}}{(1 - |x|^2)^{1/2}} + \dfrac{x_i x_j}{(1 - |x|^2)^{3/2}} \right) \\
    & =\dfrac{1}{(1 - |x|^2)^{3/2}} \left( \delta_{ij}(1 - |x|^2) + x_i x_j  \right)   \equiv \dfrac{1}{(1 - |x|^2)^{3/2}}\, b_{ij}(x).
\end{align*}
In addition, $B(x)= (b_{ij}(x))$ is diagonalizable with eigenvalue $1$ with multiplicity $1$, and eigenvalue $(1-|x|^2)$ with multiplicity $d-1$. Hence 
\begin{align*}
V(D \Psi(x)) & = \left | \det A(x) \right| = \left | \det \dfrac{1}{(1 - |x|^2)^{3/2}} \,B(x) \right | \\
& = (1 - |x|^2)^{-d(3/2)} (1- |x|^2)^{d-1}  = (1 -|x|^2)^{-d/2 - 1}.
\end{align*}
 Note if $\gamma$ is the angle between $\tilde{x}$ and $-e_n$, then $ \cos(\gamma)  = \tilde{x} \cdot (-e_n) = \sqrt{1 - |x|^2}$. Thus, viewing $E$ as a parametrized manifold, we have that 
      \begin{align*}
          \int_E \dfrac{K(0-y)}{|0-y|^{d+\alpha}} \; d\HD^d (y) & = \int_{U^d} K(-\tilde{x}) \left( \sqrt{1 - |x|^2}\right)^{d+\alpha} V(D \Psi(x)) \; d \HD^d(x) \\
              & = \int_{U^d} K(-\tilde{x}) \left( \sqrt{1 - |x|^2} \right)^{\alpha - 2} \; d\HD^d (x) \\
          & = \int_{U^d} K(-\tilde{x}) ( \tilde{x} \cdot (-e_n))^{\alpha -2 } \; d \HD^d(x).
       \end{align*}
      Finally, we rewrite this last integral as the integral over the parametrized manifold $H^d(E, 0)$. Recall that in this case the map $\Phi: U^d \ra H^d(E, x)$ that takes $x \ra \tilde{x}$ has
\begin{align*}
    V(D \Phi (x))     = 1/ \sqrt{ 1 - |x|^2}   = (\tilde{x} \cdot (-e_n))^{-1}.
\end{align*}
Altogether we have that 

      \begin{align*}
\int_E \dfrac{K(0-y)}{|0-y|^{d+\alpha}} \; d\HD^d (y) & =\int_{U^d} K(-\tilde{x}) ( \tilde{x} \cdot (-e_n))^{\alpha -2 } \; d \HD^d(x) \\
& = \int_{U^d} K(-\tilde{x}) ( \tilde{x} \cdot (-e_n))^{\alpha - 1 } V(D\Phi(x))  \; d \HD^d(x) \\
& = \int_{H^d(E, x)} K(-z) \left( z \cdot (-e_n) \right)^{\alpha - 1} \; d\HD^d(z). 
      \end{align*}
\end{proof}

With this change of variables, we obtain a more concrete characterization of $(d, \alpha)$-distance-exact kernels in the scale-invariant setting. Geometrically, the following statement says that the distance-exact kernels (outside $E$) are precisely those functions on the sphere whose weighted average over a certain collection of half $d$-arcs is constant.

\begin{thm}
Let $n, d \in \N, \alpha >0$ with $d < n$, and let $E \in G(n,d)$. Then a kernel $K \in L^\infty(\sphere^{n-1})$ is $(d,\alpha)$-distance-exact for $E$ if and only if 
\begin{align}
   \int_{H^d(E, x_0)}K(-A w)(w \cdot w_0)^{\alpha - 1} d\HD^d(w)= \int_{H^d(E, x_0)}K(-w)(w \cdot w_0)^{\alpha - 1} d \HD^d(w) \label{ainv}
\end{align} 
for any $x_0 \not \in E$ with $\delta_E(x_0) =1$, and every orthogonal transformation $A \in O^n$ such that $A(E) = E$. Here, as above, $w_0 = (x_0 - P_E(x_0))/|x_0-P_E(x_0)|$. \label{thm:wnda}
\end{thm}
\begin{proof}

     Fix the parameters $n,d,\alpha, E,$ as in the statement of the theorem and let $x_0 \not \in E$ be a point with $\delta_E(x_0) = 1$. Recall that since $E$ is flat we have $D_{K, E}$ is $1$-homogeneous. In particular, this implies that $K$ is $(d,\alpha)$-distance-exact for $E$ if and only if for any point $x \not \in E$ with $\delta_E(x) = \delta_E(x_0)$, we have $D_{K,E}(x) = D_{K, E}(x_0)$, or equivalently, $R_{K, E}(x) = R_{K, E}(x_0)$. By virtue of Lemma \ref{cov}, this is true if and only if 
\begin{align}
    \int_{H^d(E, x_0)} K(-w) (w \cdot w_0(x_0))^{\alpha - 1} \; d\HD^d(w) & = \int_{H^d(E, x)} K(-w) (w \cdot w_0(x))^{\alpha - 1} \; d\HD^d(w) \label{condrota}
\end{align}
for every such $x$.

Now, since $E$ is flat, $R_{K,E}$ is translation invariant with respect to vectors parallel to $E$. Hence we need only consider such $x$ with $\delta_E(x) = \delta_E(x_0)$ and such that there is an orthogonal transformation $A$ for which $Ax_0 = x$ and $A(E) = E$. Since $A_\sharp H^d(E, x_0) = H^d(E,x)$, changing variables in the right-hand side of (\ref{condrota}) (and using that $A$ is orthogonal), we see that $K$ is $(d,\alpha)$-distance-exact for $E$ if and only if 
\begin{align*}
      \int_{H^d(E, x_0)} K(-A w) (w \cdot w_0(x_0))^{\alpha - 1} \; d\HD^d(w) & = \int_{H^d(E, x_0)} K(-w) (w \cdot w_0(x_0))^{\alpha - 1} \; d\HD^d(w)
\end{align*}
for any orthogonal transformation $A$ that fixes $E$.
\end{proof}

We end with several observations on the (non-)existence of homogeneous distance exact kernels in higher co-dimension. First we observe that there are many of them in co-dimension 1. 

\begin{cor}
All even kernels give rise to $(n-1, \alpha)$-distance exact kernels for any $\alpha > 0$. Furthermore, if $\alpha = 1$ the constant $c_E$ can be taken independent of the plane $E$. \label{rmk:dist_exact_codim_1}
\end{cor}

\begin{proof} 
When $d = n-1$, the condition \eqref{ainv} is equivalent to
     \begin{align*}
         \int_{H^{n-1}(E, x_0)} K(-w) (w \cdot w_0)^{\alpha-1} \; d\HD^{n-1}(w) & = \int_{H^{n-1}(E, x_0)}K(w) (w \cdot w_0)^{\alpha-1} \; d\HD^{n-1}(w).
     \end{align*}
     In particular, any \textit{even} kernel $K \in L^\infty(\sphere^{n-1})$ satisfies this condition, and thus any such kernel will be $(d,\alpha)$-distance-exact. Moreover, in the case $\alpha = 1$, the integral is independent of the particular $E$ chosen.
\end{proof}

On the other hand, in higher co-dimension there do not exist non-trivial distance-orthogonal kernels which are $0$-homogeneous. The authors would like to thank Dmitriy Bilyk for pointing out the connection of the integral conditions on distance-exact kernels on $\sphere^2$ to the Funk Transform, which gives rise to the following result.

\begin{cor}
In $n =3$, every $0$-homogeneous $(1, 1)$-distance orthogonal kernel is identically equal to zero. \label{rmk:dist_exact_codim_2}
\end{cor} 

\begin{proof}
Suppose that $K \in C(\sphere^2), \alpha =1$, and $K$ is $(1, 1)$-distance-orthogonal. Then Theorem \ref{thm:wnda}  implies that
\begin{align}
\int_{H^1(E,x_0)} K(w) \; d\HD^1(w) = 0 \label{half_1_arcs}
\end{align}
for each half $1$-arc $H^1(E, x_0) \subset \sphere^2$, from which one can deduce that $K$ is necessarily even. To see this, simply restrict the half $1$-arcs we consider to be contained in the same great circle of $\sphere^2$. Then abusing notation and writing $K$ for the restriction of $K$ to this circle in polar coordinates, we obtain that for each $\theta \in \R$,
\begin{align*}
\int_\theta^{\pi + \theta} K(\omega) d \omega = 0.
\end{align*}
It then follows that (again, abusing notation), $K(\theta) = K(\pi + \theta)$  for each $\theta$. Lifting this to our original kernel implies that $K$ is even. 

Now from here, using \eqref{half_1_arcs} applied to the half $1$-arcs $H^1(E, x_0)$ and $-H^1(E, x_0)$ implies 
\begin{align*}
\int_C K(w) \; d\HD^d(w) = 0
\end{align*}
for each great circle $C \subset \sphere^2$. If $\mathcal{F}$ is the Funk transform on $\sphere^2$, then this says exactly that $\mathcal{F}(K) \equiv 0$. Since this transform is invertible on even continuous functions (see for example, Chapter 2 in \cite{HELGASON} for references and the definition of the Funk transform), we therefore conclude that $K \equiv 0$.
\end{proof}

\bibliographystyle{alpha}
\bibliography{bibl}
\end{document}